\documentclass[a4paper,10pt]{amsart}
\usepackage{color,latexsym,amsthm,psfrag,amsfonts,amsmath,amssymb,epsfig
}

\usepackage{hyperref}

\DeclareMathOperator{\Cox}{Cox}
\renewcommand{\div}{\textrm{div}}

\makeatletter
\@addtoreset{table}{section}
\makeatother

\begin{document}

\thispagestyle{empty}

\newtheorem{theorem}{Theorem}[section]
\newtheorem{lemma}[theorem]{Lemma}
\newtheorem{claim}[theorem]{Claim}
\newtheorem{cor}[theorem]{Corollary}
\renewcommand{\proofname}{Proof}
\newtheorem{property}[theorem]{Property}
\newtheorem{prop}[theorem]{Proposition}

\theoremstyle{definition}
\newtheorem{defin}[theorem]{Definition}
\newtheorem{question}[theorem]{Question}
\newtheorem{remark}[theorem]{Remark}
\newtheorem{example}[theorem]{Example}
\newtheorem{notation}[theorem]{Notation}

\def \H{{\mathbb H}}
\def \N{{\mathbb N}}
\def \R{{\mathbb R}}
\def \X{{\mathbb X}}
\def \E{{\mathbb E}}
\def \Z{{\mathbb Z}}
\def \S{{\mathbb S}}
\def \G{{\mathcal G}}
\def \Gr{{\mathrm Gr}}
\def \Vol{{\mathrm Vol}}
\def \l{\langle }
\def \r{\rangle }
\def \[{[ }
\def \]{] }
\def \d{D\,}
\def \sign{\text{\,sign\,}}
\def \conv{\text{\,conv\,}}
\def \wt{\widetilde}
\def \wh{\widehat}
\def \a{\alpha}
\def \b{\beta}
\def \dim{\text{\,dim\,}}
\def \rank{\text{\,rank\,}}
\def \Symb{\text{\,Simb\,}}
\def \relint{\text{rel\.int\,}}

\newcommand{\arcsinh}{\mathop{\mathrm{arcsinh}}\nolimits}
\newcommand{\vn}{\mathop{\mathrm{int}}\nolimits}
\newcommand{\rel}{\mathop{\mathrm{rel\:int}}\nolimits}
\newcommand{\w}{\widetilde }
\renewcommand{\o}{\overline }
\newcommand{\Q}{\mathbb{Q}}

\newcommand{\coR}{\color{red}}
\newcommand{\coG}{\textcolor{green}}
\newcommand{\coB}{\textcolor{blue}}

\newcommand{\<}{\left<}
\renewcommand{\>}{\right>}
\newcommand{\abs}[1]{\left|#1\right|}


\author{Anna Felikson}
\address{Department of Mathematical Sciences, Durham University, Science Laboratories, South Road, Durham DH1 3LE, UK}
\email{anna.felikson@durham.ac.uk, pavel.tumarkin@durham.ac.uk}

\author{Jessica Fintzen}
\address{Department of Mathematics, Harvard University, One Oxford Street, Cambridge, MA 02138, USA}
\email{fintzen@math.harvard.edu} 

\author{Pavel Tumarkin}

\thanks{Research partially supported by RFBR grant 11-01-00289-a (P.~T.) and DFG grant FE-1241/2-1 (A.~F.)}

\title{Reflection subgroups of odd-angled Coxeter groups}



\begin{abstract} 
We give a criterion for a finitely generated odd-angled Coxeter group to have a proper finite index subgroup generated by reflections. The answer is given in terms of the least prime divisors of the exponents of the Coxeter relations.

\end{abstract}

\maketitle

\tableofcontents

\section{Introduction}

Reflection subgroups of Coxeter groups arise in various contexts. As proved by Dyer~\cite{Dyer} and Deodhar~\cite{Deodhar}, subgroups of Coxeter groups generated by reflections are Coxeter groups themselves. In the arithmetic over $\Q$ case they provide regular subalgebras of corresponding Kac-Moody algebras (see~\cite{Dynkin,max-hyperbolic,FN,affine,hyperbolic}).

Classifications of reflection subgroups of Coxeter groups are known in some special cases. Namely, the reflection subgroups of finite~\cite{Carter,Dynkin} and affine~\cite{Dyer,euclidean,DyerLehrer} Coxeter groups are completely classified. For reflection groups in the real hyperbolic space, there is a classification~\cite{hyperbolic} of reflection subgroups in the case of both the group and the subgroup having simplicial fundamental domains (the simplices may have distinct dimensions).     

Of special interest are reflection subgroups of finite index. In the arithmetic case, they correspond to those regular subalgebras of Kac-Moody algebras that have maximal rank. In the case of reflection groups acting on a space of constant curvature, a finite index reflection subgroup provides a tessellation of its fundamental polytope by copies of the fundamental polytope of the group. The same picture can be observed on the Davis complex of a general Coxeter group with finite index reflection subgroup~\cite{FT}.

In this paper, we solve the general problem of existence of finite index reflection subgroups in odd-angled Coxeter groups, i.e. in the groups with all orders $m_{ij}$ of products $s_is_j$ of generating reflections being odd (see Section~\ref{pril}).
The answer is given in terms of a {\it divisibility diagram} $\Cox_{\div}(W)$ which is a modification of the Coxeter diagram:
the edges of  $\Cox_{\div}(W)$ are labeled by least (non-trivial) divisors of $m_{ij}$ instead of $m_{ij}$ themselves, and the vertices are not joined if $m_{ij}=\infty$. The connectivity of the divisibility diagram of an odd-angled Coxeter group is equivalent to the existence of one conjugacy class containing all reflections of the group. 

We approach {{the}} problem by considering {\it special subgroups} of the given Coxeter group 
(they are also called {\it standard parabolic} in the literature) generated by a subset of the initial generating set (see Section~\ref{pril} for the precise definitions). The main tool relating finite index reflection subgroups in a Coxeter group and in its special subgroups is the following {\it Subdiagram Lemma}:

\setcounter{section}{3}
\setcounter{theorem}{1}
\begin{cor}[Subdiagram Lemma]
Let $W$ be a Coxeter group with set of generators $S$ such that $\Cox_{\div}(W)$ is connected. Suppose that $V\subsetneq W$ is a reflection subgroup of index $n$, $1 < n < \infty$. Let $W_1$ be a special subgroup of $W$. Then $W_1$ contains a proper reflection subgroup of index at most $n$.

\end{cor}

The Subdiagram Lemma implies that the divisibility diagrams of Coxeter groups with proper finite index reflection subgroups compose a partially ordered set (with the order being inclusion), which means that it is sufficient to classify minimal (by inclusion) divisibility diagrams defining Coxeter groups without subgroups. This is done in Theorem~\ref{minimal}, the result is shown in Table~\ref{answer}.

\medskip
\noindent
The groups with disconnected divisibility diagrams are treated based on the following lemma:

\setcounter{section}{6}
\setcounter{theorem}{2}
\begin{lemma}
Let $W$ be an odd-angled Coxeter group and $W=W_1* W_2*\dots* W_k$. 
Then  $W$ contains a proper finite index reflection subgroup if and only if at least one of $W_1,\dots,W_k$ contains one. 

\end{lemma}

This brings us to the following criterion.  

\begin{theorem}
An odd-angled Coxeter group $W$ contains no finite index proper reflection subgroup if and only if each connected component of $\Cox_{\div}(W)$ contains one of the diagrams shown in Table.~\ref{answer} as a subdiagram.

\end{theorem}

We can also reformulate the criterion to get a self-contained form of the statement, i.e., to avoid references to the table of minimal groups without subgroups.

\begin{cor}
An odd-angled Coxeter group $W$ contains a finite index proper reflection subgroup if and only if 
$\Cox_{\div}(W)$ contains at least one connected component $C$ of one of the following three types:

\begin{itemize}
\item[1.] the order of $C$ is $1$ or $2$;

\item[2.] $C$ contains at most one multiple edge;

\item[3.] $C$ contains a subdiagram $D$ of order $3$ with labels $(5,5,3)$, and every non-absent edge of $C$ except the edges of $D$ is simple.

\end{itemize}

\end{cor}

\setcounter{section}{1}

\medskip

The paper is organized as follows.
In Section~\ref{pril}, we recall necessary facts about Coxeter groups and their Davis complexes.
Section~\ref{sub} is devoted to the proof of the Subdiagram Lemma (Corollary~\ref{subd}).
In Section~\ref{examples}, we construct examples of finite index reflection subgroups in two series of odd-angled Coxeter groups.
Section~\ref{abs} is devoted to the proof of absence of finite index reflection subgroups in most Coxeter groups with 
connected divisibility diagram. The combinatorial tools necessary for the proof are developed in Section~\ref{tools}.
Finally, in Section~\ref{min} we combine the results of the previous two sections to obtain the list (Table~\ref{answer})
of minimal groups containing no finite index reflection subgroups (see Theorem~\ref{minimal}).
We also prove Lemma~\ref{component} (concerning disconnected divisibility diagrams) 
and use it to prove Theorem~\ref{cor1}.

{We note that some of the technical tools and partial results (especially in Section~\ref{abs}) still hold if we consider a larger class of groups (namely, {\it skew-angled} Coxeter groups, where $m_{ij}$ may be even but not equal to $2$), {and some even for arbitrary Coxeter groups}. However, already in rank $3$ there are series of examples of finite index subgroups of skew-angled Coxeter groups (see Remark~\ref{evenseries}) indicating that the answer for skew-angled groups will be more complicated. }


\section{Preliminaries}
\label{pril}

In this section, we will mainly follow~\cite{Davis} to reproduce definitions and essential properties of Coxeter groups and related constructions.

\subsection{Coxeter system}
A group $W$ is called a \textit{Coxeter group} if it has a representation of the form 
$$	W=\< S | (s_is_j)^{m_{ij}}=1 \, \forall \, s_i, s_j \in S \>
$$
where $S$ is a set, $m_{ii}=1$ and $m_{ij} \in \N_{>1} \cup \{\infty \}$ for all $i\neq j$. Thereby $m_{ij}=\infty$ means that there is no relation on $s_is_j$. Furthermore, throughout the paper we require $S$ to be finite. 
 A pair $(W,S)$ of a Coxeter group $W$ and its set of generators $S$ is called a \textit{Coxeter system}. The cardinality of $S$ is called {\it rank} of the Coxeter system.

An element of $W$ is said to be a \textit{reflection} if it is conjugated in $W$ to an element of $S$.
A {\it reflection subgroup} in a Coxeter group is a proper subgroup generated by reflections{, where ``proper''  means of index greater than one}.
A {\it special subgroup} of $(W,S)$ is a reflection subgroup generated by elements of $S'$ where $S'\subset S$.

A Coxeter group is called {\it skew-angled} if $m_{ij}\ne 2$ for every pair $(i,j)$ and {\it odd-angled} if all $m_{ij}$ are odd or infinite.

\medskip

Coxeter groups are usually presented by Coxeter diagrams (see~\cite{V}).
In this paper it will be convenient to  use the following modification of Coxeter diagrams.

\begin{defin}[Divisibility diagram]
\label{div}
Let $(W,S)$ be a Coxeter system. A {\it divisibility diagram}  $\Cox_{\div}(W)$ of $W$ is a one-dimensional simplicial complex with edges labeled by positive integers constructed in the following way:
\begin{itemize}
\item vertices $v_i$ of $\Cox_{\div}(W)$ correspond to generating reflections $s_i\in S$;
\item vertex $v_i$ is joined with vertex $v_j$ by an edge labeled by $k>1$ if $k$ is the minimal non-trivial divisor of $m_{ij}$ {{(as in Coxeter diagrams, label $k=3$ is omitted)}}; 
\item $v_i$ and $v_j$ are  not joined if the order of $(s_is_j)$ is infinite. 

\end{itemize} 
\end{defin}

We call an edge without any label {\it simple}, and all the other edges {\it multiple}. If two vertices are not joined, we say they are joined by an {\it absent edge}.   
A divisibility diagram $\Cox_{\div}(W)$ of a skew-angled Coxeter group can be obtained from the Coxeter diagram in the following way: substitute all labels by their least prime divisors, and delete all dashed edges {{(corresponding to infinite dihedral subgroups)}}. 

{{By a {\it subdiagram} of a divisibility diagram (or a Coxeter diagram) we always mean ``full'' subdiagram, i.e., a diagram obtained by removing some vertices and all edges emanating from them.}}

\subsection{Length function and Exchange Condition}
Since a Coxeter group $W$ is generated by the elements of $S$, we can write each $w \in W$ in the form $s_1 s_2 \cdots s_k$, where $s_1, s_2, \hdots, s_k$ are some (not necessarily distinct) elements of $S$. If $k$ is chosen such that $w$ cannot be written as a product of less than $k$ elements of $S$, we call $s_1 s_2 \cdots s_k$ a \textit{reduced expression} for $w$ and say that $k$ is the \textit{length} of $w$, which we denote by $l(w)$. By convention $l(1)=0$.

The proof of the following fundamental result on the length function for Coxeter groups can be found in~\cite[Theorem 5.8]{Humphreys}.
\begin{theorem}[Strong Exchange Condition]
Let $(W,S)$ be a Coxeter system and $w=s_1 s_2 \cdots s_n$ with $s_1, s_2, \hdots, s_n$ being not necessarily distinct elements of $S$. If $t$ is a reflection satisfying $l(wt)<l(w)$, then there exists $i \leq n$ for which $wt=s_1 \cdots \hat s_i \cdots s_n$, where $\hat s_i$ denotes that the element $s_i$ is omitted. 
\end{theorem}

If the element $t$ in the above theorem is required to be contained in $S$ and the word $s_1\cdots s_n$ is reduced, the resulting weaker statement is known as Exchange Condition. The following direct corollary will be useful later in this paper.
\begin{cor} \label{ending-element-corollary}
Let $(W,S)$ be a Coxeter system, $w\in W$ and $t \in S$. If $l(wt)<l(w)$, then there exists a reduced expression of $w$ that ends in $t$. 
\end{cor}

To prove the corollary take a reduced expression $s_1\cdots s_n$ for $w$ and multiply the equation $wt=s_1 \cdots \hat s_i \cdots s_n$ by $t$ from the right.

 As an element of $W$ might be represented by many different expressions, a natural question to ask is when two expressions represent the same element. The following theorem due to Tits~\cite{Tits}, whose proof is based on the Exchange Condition (see~\cite[Theorem 3.4.2]{Davis}), provides us with an algorithm to solve this question.

\begin{theorem} \label{M-operation}
	An expression $w=s_1 s_2 \cdots s_n$ with $s_1, \hdots, s_n \in S$ is  reduced if and only if it cannot be shortened by a sequence of the following two operations {{(called {\it M-operations})}}:
	\begin{enumerate}
	\item Delete $s_i s_{i+1}$ with $s_i=s_{i+1}$, $1 \leq i <n$, i.e. $w=s_1 s_2 \cdots \hat s_i \hat s_{i+1} \cdots s_n$.
	\item Replace $s_i s_{i+1} \cdots s_{i+m_{ij}-1}$ with $s_{i+2m}=s_i$ for $1 \leq m \leq (m_{ij}-1)/2$ and $s_{i+2m-1}=s_j$ for $1 \leq m \leq m_{ij}/2$ by $s_j s_{i} s_{i+1} \cdots s_{i+m_{ij}-2}$.
	\end{enumerate}
	Moreover, two reduced expressions $s_1 s_2 \cdots s_n$ with $s_1, \hdots, s_n \in S$ and $t_1 t_2 \cdots t_n$ with $t_1, \hdots, t_n \in S$ represent the same element in $W$ if and only if they can be transformed into each other by a sequence of operations of the second type.
\end{theorem}
We have already observed that if $l(wt)<l(w)$ for a generator $t \in S$, then there exists a reduced expression of $w$ that ends in $t$. The following {{property}} of the set of possible last letters of reduced expressions for a given element of $W$ {{will be used in the sequel}}. 

\begin{lemma}[\cite{Davis},~Lemma 4.7.2] 
\label{end-of-reduced-expression}
 Let $w \in W$ and denote by $In(w)$ the subset of $S$ in which a reduced expression of $w$ can end. Then the subgroup generated by $In(w)$ is finite. 
\end{lemma}

\subsection{Davis complex}

For any Coxeter system $(W,S)$ there exists a contractible
piecewise Euclidean cell complex $\Sigma(W,S)$ (called {\it Davis complex}) 
on which $W$ acts discretely, properly and cocompactly. 
The construction was introduced by Davis~\cite{Davis0}.
In~\cite{M} Moussong proved that this complex yields
a natural complete piecewise Euclidean metric which is $CAT(0)$. 
We give a brief description of this complex following~\cite{NV}.

For a finite group $W$ the complex $\Sigma(W,S)$ is just one cell, which is obtained
as the convex hull $C$ of the $W$-orbit of a suitable point $p$ in the standard linear 
representation of $W$ as a group generated by reflections. The point $p$ is chosen in such 
a way that its stabilizer in $W$ is trivial and all the edges of $C$ are of length $1$.
The faces of $C$ are naturally identified with Davis complexes of the subgroups of $W$
conjugated to special subgroups.

If $W$ is infinite, the complex $\Sigma(W,S)$ is built up of the Davis complexes of maximal 
finite subgroups of $W$ by gluing them together along their faces corresponding to common 
finite subgroups. The $1$-skeleton of $\Sigma(W,S)$ considered as a combinatorial graph 
is isomorphic to the Cayley graph of $W$ with respect to the generating set $S$. 

In what follows, if $W$ and $S$ are fixed, we write $\Sigma$ instead of $\Sigma(W,S)$.

\subsection{Walls and convex polytopes} \label{walls-section}

The group $W$ admits a natural action on $\Sigma(W,S)$ by reflections. The action is an isometry with respect to $CAT(0)$ piecewise Euclidean metric. A \textit{wall} $H_w$ corresponding to a reflection $w\in W$ is the fixed point set of $\Sigma(W,S)$ under the action of $w$. {{In particular, two walls $H_w$ and $H_u$ intersect if and only if the dihedral group generated by $u$ and $w$ is finite.}} Any wall divides $\Sigma$ into two connected components. We denote their closures by $H_w^+$ and $H_w^-$ and call them \textit{halfspaces}. Walls are totally geodesic, i.e. any geodesic between two points contained in the same wall lies entirely in this wall, see~\cite{Noskov}. This implies that every intersection of walls is, in its turn, totally geodesic, and halfspaces are convex. We note that since  $\Sigma(W,S)$ is $CAT(0)$, there is a unique geodesic through every two points of $\Sigma(W,S)$.

Following~\cite{FT}, we call an intersection of finitely many halfspaces not contained in any wall~\textit{convex polytope}.  In the sequel writing $P=\bigcap\limits_{i \le n} H_{w_i}^+$ we always assume that $P$ cannot be defined by less than $n$ walls.

A convex polytope that does not contain any other convex polytope is called a \textit{chamber}. Chambers are fundamental domains of the action of $W$ on $\Sigma$, and each chamber contains precisely one vertex of the $1$-skeleton of $\Sigma$
which corresponds to an element of the group $W$. Following~\cite{Noskov}, we will denote the chamber of $\Sigma$ corresponding to $w\in W$ by $D(w)$. In the sequel, by abuse of notation, we will sometimes identify the chamber corresponding to $w \in W$ with $w$ itself.

We call two elements $w$ and $v$ in $W$ \textit{neighbors} if and only if there exits $s \in S$ such that $w = vs$ (this corresponds to two neighboring chambers of $\Sigma$). Then a \textit{gallery of length $k$} is a sequence of chambers $D(w_1), D(w_2), \hdots, D(w_k)$, $w_1, \hdots, w_k \in W$, such that for every $1\le i \le k-1$ the elements $w_i$ and $w_{i+1}$ are neighbors. A gallery is {\it geodesic} if its length is equal to $l(w_1^{-1}w_k)+1$. Note that for any two chambers there exists a geodesic gallery between them. 

{The following lemma is well-known (see e.g.~\cite[Lemma 2.5(i)]{R} or~\cite[Lemma~2.2.5]{Noskov}).} 

\begin{lemma}\label{geodesic-gallery-lemma}
A gallery is geodesic if and only if it crosses any wall at most once.
\end{lemma}

Let $P=\bigcap\limits_{i\le n} H_{w_i}^+$ be a convex polytope and $I$ be a subset of $\{1,2, \hdots , n\}$. If the intersection $\bigcap\limits_{i \in I} H_{w_i} \cap P$ is nonempty, it is called a \textit{face} of $P$. The intersection of $P$ with one of the walls $H_{w_i}$  is called a \textit{facet} of $P$. The walls $H_{w_i}$ itself will be called \textit{defining walls} of $P$. 

Following \cite{FT}, we define the \textit{dihedral angle} formed by two intersecting walls $H_{w_1}$ and $H_{w_2}$ to be the Euclidean dihedral angle between $H_{w_1} \cap C$ and $H_{w_2} \cap C$ in $C$, where $C$ is a maximal cell of $\Sigma$ containing $H_{w_1} \cap H_{w_2}$. 
We define analogously the \textit{dihedral angle} formed by two intersecting facets $f$ and $g$ of a convex polytope $P$ to be the dihedral angle between $f \cap \widetilde C$ and $g \cap \widetilde C$, where $\widetilde C$ is a cell of $\Sigma$ such that $\widetilde C \cap f \cap g$ is nonempty. If all the dihedral angles between any two intersecting facets of a convex polytope $P$ are less than or equal to $\pi/2$, we say that the convex polytope $P$ is \textit{acute-angled}. A convex polytope is a {\it Coxeter polytope} if all its dihedral angles are integer submultiples of $\pi$. 
Note that a chamber is a Coxeter polytope with angles $\pi/m_{ij}$.
{
\begin{remark}
For Coxeter groups $W$ acting cocompactly on the Euclidean space $\E^n$ or the hyperbolic space $\H^n$ the cell complex structure of the Davis complex of $W$ can be naturally identified with the cell complex constructed from fundamental polytopes of the $W$-action on $\E^n$ or $\H^n$. In particular, this holds for rank $3$ groups with $\frac{1}{m_{12}}+ \frac{1}{m_{13}}+\frac{1}{m_{23}}<1$, $m_{ij}\ne\infty$. If some of $m_{ij}$ are infinite in a rank $3$ group $W$, then $W$ acts on the hyperbolic plane $\H^2$ with a fundamental triangle $P$ of finite volume, and the combinatorics (but not the topology!) of the tessellation of $\H^2$ by copies of $P$ coincides with the combinatorics of the Davis complex of $W$.   

\end{remark}
}

\section{Subdiagram Lemma}
\label{sub}

In this short section we prove the Subdiagram Lemma, which states that the property of having a finite index reflection subgroup
is preserved when we take a special subgroup of the group.

Let $W$ be an arbitrary finitely generated Coxeter group, let $\Sigma(W)$ be the Davis complex of $W$ and $D(1)\subset \Sigma(W)$ be the fundamental chamber of $W$ corresponding to the identity element $1$ of $W$.

Let $V \subset W$ be a reflection subgroup, and denote by $R'$ the set of reflections in $V$. Consider $\qquad\qquad\Sigma(W)\setminus\bigcup_{w \in R'} H_w$. The closure of each connected component of this space in $\Sigma(W)$ is a fundamental domain for the action of $V$ on $\Sigma(W)$. We call the closure of the connected component that contains the identity \textit{principal fundamental domain}. 
Note that the principal fundamental domain is a Coxeter polytope. Conversely, every Coxeter polytope in $\Sigma(W)$ is a fundamental domain of a reflection subgroup of $W$.

\begin{lemma}
\label{special subgroup}
Let $W$ be a Coxeter group with set of generators $S$ such that {{all elements of $S$ are conjugate in $W$}}. Suppose that $V\subsetneq W$ is a reflection subgroup of index $n$, $1 < n < \infty$. Let $W_1$ be a special subgroup of $W$. Then $W_1$ contains a reflection subgroup of index at most $n$.
\end{lemma}

\begin{proof}
Since $V \subsetneq W$ there exists $s_i \in S$ such that $s_i \notin V$. Let $S_1=S \cap W_1$ be a set of generators for $W_1$, and let $s \in S_1$. {{By the assumptions of the lemma, all the generating reflections are conjugate, so all the reflections are conjugate. Thus, we can find $w \in W$ such that $s=ws_iw^{-1}$.}} Then $V'=wVw^{-1}$ is a non-trivial finite index reflection subgroup of $W$. Denote its principal fundamental domain by $F_{V'}$. 

Define $I$ to be $\{w \in W | D(w) \in F_{V'} \}$. Set $I_1=I \cap W_1$. Then $\abs{I_1} \leq \abs{I}=n$, and $F_1 = \bigcup_{w \in I_1} D(w) \subset \Sigma(W_1,S_1)$  is a Coxeter polytope (as $F_{V'}$ is a Coxeter polytope, and the reflections corresponding to the walls of $F_1$ form a subset of the reflections corresponding to the walls of $F_{V'}$) that contains at most $n$ chambers. Note that $s \notin V'$ by the definition of 
$V'$. Hence $D(s) \in F_{V'}$, and therefore $D(s) \in F_1$. Thus, $F_1$ contains at least two chambers ($D(1)$ and $D(s)$), i.e. it is the fundamental domain for a reflection subgroup of $W_1$ of index at most $n$. 
\end{proof}

{
Throughout the paper we will use the following corollary of Lemma~\ref{special subgroup} (we will refer to it as {\em Subdiagram Lemma}). 
}

{
\begin{cor}
\label{subd}
Let $W$ be an odd-angled Coxeter group with set of generators $S$ such that $\Cox_{\div}(W)$ is connected. Suppose that $V\subsetneq W$ is a reflection subgroup of index $n$, $1 < n < \infty$. Let $W_1$ be a special subgroup of $W$. Then $W_1$ contains a {proper} reflection subgroup of index at most $n$.

\end{cor}
}

{ 
\begin{proof}
Two generating reflections $s_0,s_k\in S$ are conjugate in $W$ if and only if there is a sequence $s_0,s_1,\dots,s_{k-1},s_k$ of elements of $S$ such that for every $i\in[0,k-1]$ the orders $m_{ii+1}$ are odd (see e.g~\cite[Lemma 3.3.3]{Davis}). Therefore, Lemma~\ref{special subgroup} is applicable if the following ``odd version'' of the Coxeter diagram of $(W,S)$ is connected: remove from the Coxeter diagram all edges labeled by even numbers or by infinity. In the case of odd-angled groups, this diagram coincides with  $\Cox_{\div}(W)$. 
\end{proof}
}

\section{Examples of subgroups in odd-angled groups}
\label{examples}
In this section we {construct} three series of examples of finite index reflection subgroups.
Combined with the results of Section~\ref{abs}, this will provide the classification of all odd-angled groups with finite index reflection subgroups (see also Section~\ref{min}).

\begin{lemma}
\label{553}
Let $k_{13}$ and $k_{23}$ be positive integers not divisible by $2$ or $3$. Then the group $\qquad\qquad$ $W=\langle s_1,s_2,s_3 \ | \ s_i^2=(s_1s_2)^{3k_{12}}=(s_1s_3)^{5k_{13}}=(s_2s_3)^{5k_{23}}=1\rangle$ has a finite index reflection subgroup.
\end{lemma}

\begin{proof}
In order to show the existence of a finite index reflection subgroup, it suffices to construct a Coxeter polytope in $\Sigma(W,S)$ that contains finitely many, but at {least} $2$ chambers. This polytope is then the fundamental domain of a finite index reflection subgroup. If $k_{12}=1$, a possible choice for the fundamental chamber $P$ of the subgroup is shown in Fig~\ref{ex553}, right: it is obtained by gluing all rotation images of the shaded domain. 

\begin{figure}[!h]
\begin{center}
\psfrag{5}{\scriptsize $5$}
\psfrag{1}{\scriptsize $D(1)$}
\psfrag{s1}{\scriptsize $D(s_1)$}
\psfrag{s2}{\scriptsize $D(s_2)$}
\psfrag{1_}{\tiny $\pi\!/\!k_{23}$}
\psfrag{2}{\tiny $\pi\!/\!5k_{13}$}
\psfrag{3}{\tiny $\pi\!/\!k_{12}$}
\psfrag{4}{\tiny $\pi\!/\!5k_{23}$}
\psfrag{5_}{\tiny $\pi\!/\!k_{13}$}
\psfrag{6}{\tiny $\pi\!/\!3k_{12}$}
\psfrag{7}{\tiny $\pi\!/\!3k_{12}$}
\psfrag{p1}{\scriptsize 1}
\psfrag{p2}{\scriptsize 2}
\psfrag{p3}{\scriptsize 3}
\raisebox{23pt}{\epsfig{file=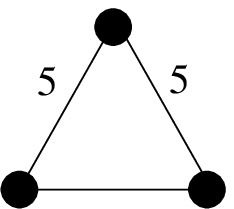,width=0.16\linewidth}}
\qquad \qquad \qquad
\epsfig{file=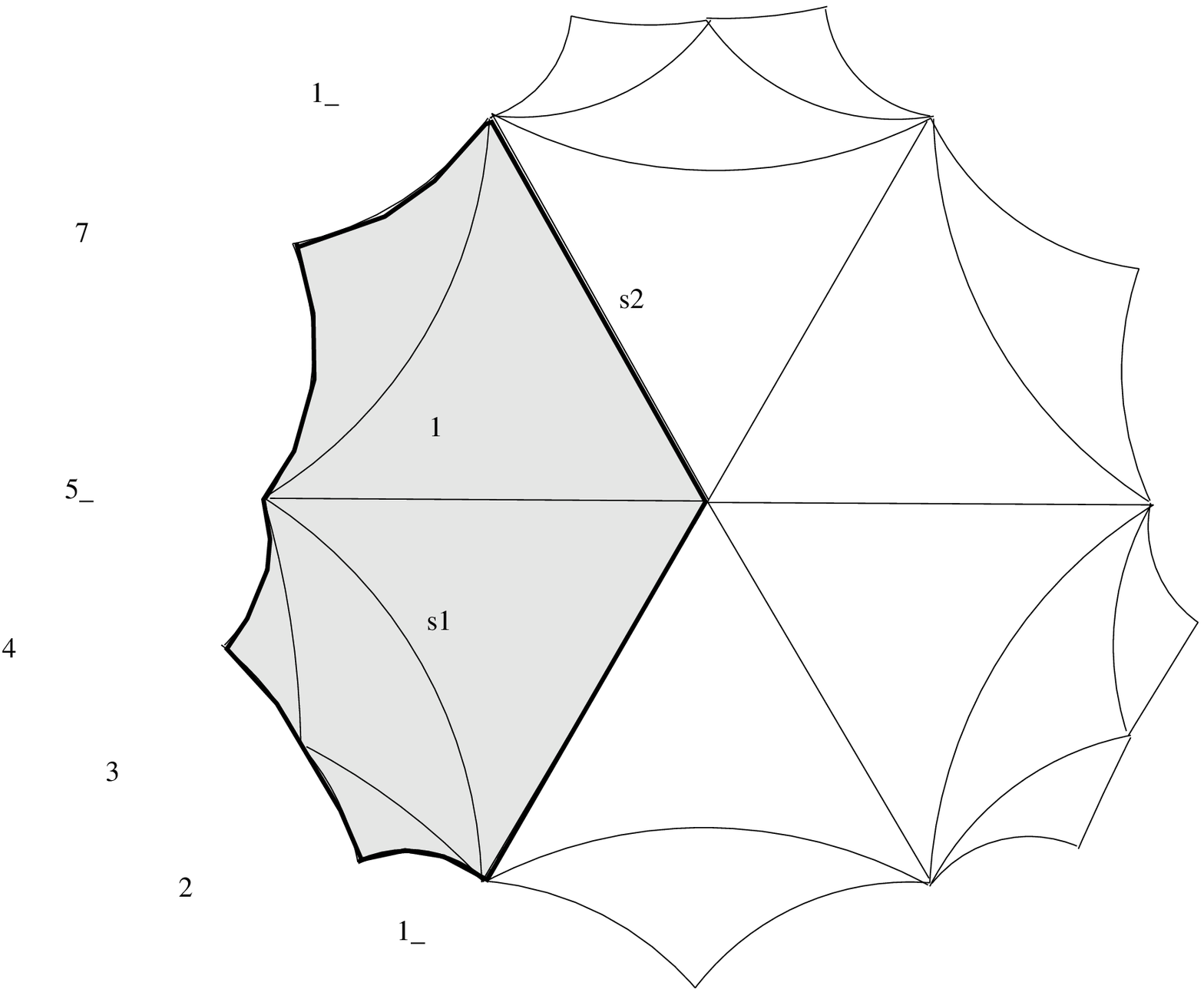,width=0.30\linewidth}
\caption{$\Cox_{\div}(W)$, and a fundamental chamber for an index $18$ subgroup of $W$ in case $k_{12}=1$}
\label{ex553}
\end{center}
\end{figure}

For the general case let (the shaded domain) $P'$ be the union of the chambers $D(1)$, $D(s_3$), $D(s_1)$, $D(s_1s_3)$, $D(s_1s_3s_1)$, $D(s_1s_3s_2)$, and take $P$ as the union of (the rotations) $(s_2s_1)^kP'$ for $0 \le k < 3k_{12}$. Then $P$ is the union of $18k_{12}$ chambers. 

For $k_{12},k_{13},k_{23}>1$ the polytope $P$ is an ($18k_{12}$)-gon. Going around the polygon counterclockwise \quad (see Fig.~\ref{ex553}), the angles are subdivided into $6$-tuples and the values in every $6$-tuple are $\left(\frac{\pi}{3k_{12}},\frac{\pi}{k_{13}},\frac{\pi}{5k_{23}},\right.$ $\left.\frac{\pi}{k_{12}},\frac{\pi}{5k_{13}},\frac{\pi}{k_{23}}\right)$. As these are submultiples of $\pi$, $P$ is indeed a Coxeter polytope, {so we obtain a reflection subgroup of index $18k_{12}$}. 

If some of $k_{12},k_{13}$ or $k_{23}$ are equal to 1, then some of the above angles are $\pi$, i.e. the number of vertices of the polygon $P$ decreases, but $P$ is still a Coxeter polytope. {More precisely, every angle of size $\frac{\pi}{k_{13}}, \frac{\pi}{k_{23}}$ or $\frac{\pi}{k_{12}}$ appears exactly $3k_{12}$ times, so with every of the numbers $k_{12},k_{13},k_{23}$ equal to one the number of angles decreases by $3k_{12}$. For example, if $k_{12}=k_{13}=k_{23}=1$, then the convex polytope $P$ is a polygon with six angles $\pi/5$ and three angles $\pi/3$.}
\end{proof}

{
\begin{remark}
\label{evenseries}
If we allow the angles to be even, the construction from Lemma~\ref{553} gives rise to many other series of finite index reflection subgroups. For example, a group with presentation
$$W=\langle s_1,s_2,s_3 \ | \ s_i^2=(s_1s_2)^{10k_{12}}=(s_1s_3)^{3k_{13}}=(s_2s_3)^{7k_{23}}=1\rangle$$
has a reflection subgroup of index $27k_{12}$ (see Fig.~\ref{3710}). This example suggests that, if angles are allowed to be even, the answer depends on the more subtle interplay of differenet divisors of the orders $m_{ij}$ of the products of generators.       
\end{remark}
}

\begin{figure}[!h]
\begin{center}
\psfrag{5}{\scriptsize $5$}
\psfrag{1}{\scriptsize $D(1)$}
\psfrag{s1}{\scriptsize $D(s_1)$}
\psfrag{s2}{\scriptsize $D(s_2)$}
\epsfig{file=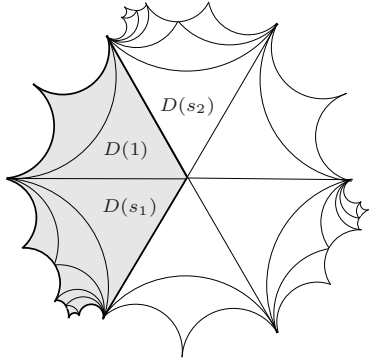,width=0.30\linewidth}
\caption{{A fundamental chamber for an index $27$ subgroup of $W$ in case $k_{12}=1$, see Remark~\ref{evenseries}}}
\label{3710}
\end{center}
\end{figure}

\begin{lemma}
\label{k}
If at most one edge of $\Cox_{\div}(W)$ is neither simple nor absent then $W$ has a finite index  reflection subgroup.

\end{lemma}

\begin{proof}
{ If $\Cox_{\div}(W)$ has only one vertex the statement is obvious. If all edges are absent the statement is also trivial: for every $s\in S$ the union $D(1)\cup D(s)$ is a Coxeter polytope. 

Let $v_1$ and $v_2$ be vertices of $\Cox_{\div}(W)$ joined by a multiple edge, if any (if all non-absent edges are simple, take any pair of vertices joined by a simple edge).}
Let $s_1$ and $s_2$ be the corresponding reflections and $W_{12}=\langle s_1,s_2\rangle$ be the special subgroup generated by these reflections. 
Consider the polytope $$P_1=\bigcup\limits_{w\in W_{12}} D(w).$$
Let $\angle\alpha \beta$ be a dihedral angle of $P_1$ (formed by the facets $\alpha$ and $\beta$). 
It can be  of one of the following two types:

1) both $\alpha$ and $\beta$ contain facets of one copy $D(w)$ of $D(1)$; then $\angle\alpha \beta$ is a submultiple of $\pi$
(since $D(w)$ is a Coxeter polytope);

2)  $\alpha$ and $\beta$ are not contained in defining walls of the same $D(w)$; then they are contained in walls of two adjacent copies of the fundamental chamber, and the dihedral angle is $2\pi/3k_{i1}$ or $2\pi/3k_{i2}$, where $3k_{ij}=m_{ij}$ for $j=1$ or $2$.

Now, to construct a Coxeter polytope, we need to add some additional copies of $D(1)$.
Namely, to each $D(w)$ such that $w\in W_{12}$ is of even length (we denote this { subgroup} of $W_{12}$ by $W_{12}^+$), we glue all chambers adjacent to $D(w)$: define 
$$P=P_1\bigcup \left( \bigcup \limits_{w\in W_{12}^+} P_w    \right),$$
where $P_w$ is the union of all chambers adjacent to $D(w)$.
It is easy to see that $P$ is a Coxeter polytope (each of its dihedral angles either belongs to one chamber or is dissected into exactly 3 parts
of size $\pi/3k_{ij}$ for some $i,j$), see Fig.~\ref{ex_k} for an example.  { It is also clear that $P$ contains finitely many (but at least two) chambers, so it defines a finite index reflection subgroup.} 
\end{proof}

\begin{figure}[!h]
\begin{center}
\psfrag{k}{\scriptsize $k$}
\epsfig{file=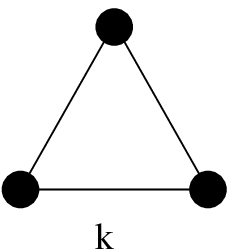,width=0.14\linewidth}
\qquad \qquad
\psfrag{k1}{\scriptsize $k$}
\psfrag{3k2}{\scriptsize $$}
\psfrag{3k3}{\scriptsize $$}
\epsfig{file=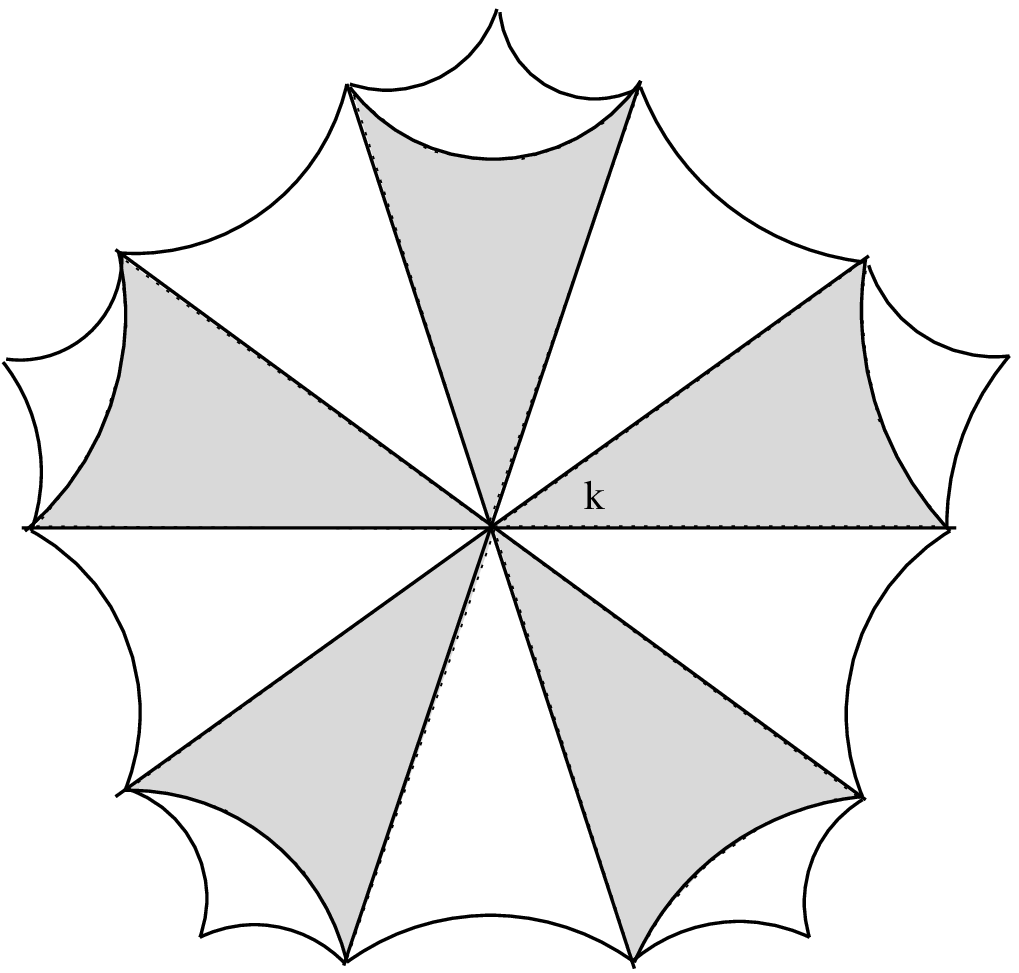,width=0.175\linewidth}
\caption{Example: $\Cox_{\div}(W)$ for a rank $3$ group and fundamental chamber $P$ constructed in the proof of Lemma~\ref{k} for the case $k=5$.}
\label{ex_k}
\end{center}
\end{figure}

Let us now generalize the example described in Lemma~\ref{553} to higher rank groups. Denote the diagram of the group described in Lemma~\ref{553} by $(5,5,3)$ { (here we assume $k_{12}$ to be odd)}.

\begin{lemma}
\label{55}
Suppose that $\Cox_{\div}(W)$ contains $(5,5,3)$ as a subdiagram and all the remaining edges in  $\Cox_{\div}(W)$  are simple
(some vertices may not be joined). Then $W$ has a finite index  reflection subgroup.

\end{lemma}

\begin{proof}
Let $v_1v_2$ and $v_1v_3$ be the edges with label $5$ of $\Cox_{\div}(W)$. Consider the subgroup $W_{123}$ of rank $3$ generated by reflections $s_1,s_2,s_3$ corresponding to vertices $v_1,v_2,v_3$ respectively. Let $\Sigma(W_{123})\subset\Sigma(W)$ be the subcomplex spanned by all cells corresponding to subgroups of $W_{123}$. According to Lemma~\ref{553}, $W_{123}$ has a finite index reflection subgroup $V'$. Denote by $P'$ the principal fundamental chamber of $V'$ (in $\Sigma(W_{123})$) constructed in Lemma~\ref{553}. We assume without loss of generality that all the chambers of $P'$ that contain a neighboring chamber outside $P'$ correspond to group elements of odd length (see Fig.~\ref{ex553}).

Denote $I=\{w\in W_{123}\,|\,D(w)\subset P'\}$. We can consider $I$ as a subset of $W$, so we can define a polytope $P_1\subset\Sigma(W,S)$ by  
$$P_1=\bigcup\limits_{w\in I} D(w).$$
It is easy to see that $P_1$ is a convex polytope, and the angles of $P_1$ are either submultiples of $\pi$ or of the type $2\pi/3k_{ij}$ { for some integers $k_{ij}$}. Now we use the same trick as in the proof of Lemma~\ref{k}: attach to $P_1$ all the neighbors of all chambers $D(w)$ with $w$ of even length. The procedure results in a Coxeter polytope $P$. 
\end{proof}

\section{Groups without finite index reflection subgroups}
\label{abs}

\subsection{Technical tools}
\label{tools}
In this section we list technical lemmas used later to  prove the absence of finite index reflection subgroups in some groups.

We will use the following notation: $(W,S)$ is the odd-angled Coxeter system under consideration,
$V\subset W$ is a finite index reflection subgroup, $P$ is the corresponding principal fundamental domain of the $V$-action on $\Sigma(W)$, and $I$ is the set of elements of $W$ such that the corresponding chambers are contained in $P$ { (recall that $P$ is a Coxeter polytope)}. 

\begin{lemma}
\label{2copies} Suppose $s_i,s_j, s_k$ are distinct elements of $S$ such that $m_{ij}, m_{ik}, m_{kj}$ are odd or infinite. 
Then 

1) $s_i$ and $s_js_ks_j$ generate an infinite dihedral group;

{
2) the wall $H_{s_i}$ separating $D(1)$ from $D(s_i)$ does not intersect the wall $H_{s_is_js_ks_js_i}$ separating $D(s_is_j)$ from $D(s_is_js_k)$.
}
\end{lemma}

\begin{proof}
{We will show that the walls $H_{s_i}$ and $H_{s_js_ks_j}$ do not intersect, which is equivalent to the first assertion. The second assertion follows since the group generated by $s_i$ and $s_js_ks_j$ coincides with the group generated by $s_i$ and $s_is_js_ks_js_i$.}
 
Consider the special subgroup $W_1\subset W$ generated by  $s_i,s_j$ and $s_k$. We distinguish two cases.

\noindent
Case 1: $m_{ij}=m_{ik}=m_{kj}=3$. In this case $W_1$ is the group generated by reflections in the sides of an equilateral triangle in the Euclidean plane, and the Davis complex $\Sigma(W_1)$ can be identified with tessellation of the plane by triangles. Then the walls $H_{s_js_ks_j}$ and $H_{s_i}$ are parallel lines, hence they do not intersect, and the group generated by $s_i$ and $s_js_ks_j$ is infinite. 

\smallskip
\noindent
Case 2: In all the other cases, $W_1$ { can be understood as a} group generated by reflections in the sides of a hyperbolic triangle. 

Assume the walls $H_{s_js_ks_j}$ and $H_{s_i}$ have a common vertex { (see Fig.~\ref{i-jkj})}. Then, together with $H_{s_j}$ (or $H_{s_k}$), they form a triangle tessellated by chambers (note that the triangle contains more than one fundamental triangle). According to the results of~\cite{F}, there are no triangles tessellated by  hyperbolic  odd-angled Coxeter triangles. The contradiction completes the proof.  
\end{proof}

\begin{figure}[!h]
\begin{center}
\psfrag{i}{\small $H_{s_i}$} 
\psfrag{j}{\small $H_{s_j}$} 
\psfrag{k}{\small $H_{s_k}$} 
\psfrag{jkj}{\small $H_{s_js_ks_j}$} 
\psfrag{3}{\scriptsize} 
\psfrag{5}{\scriptsize} 
\epsfig{file=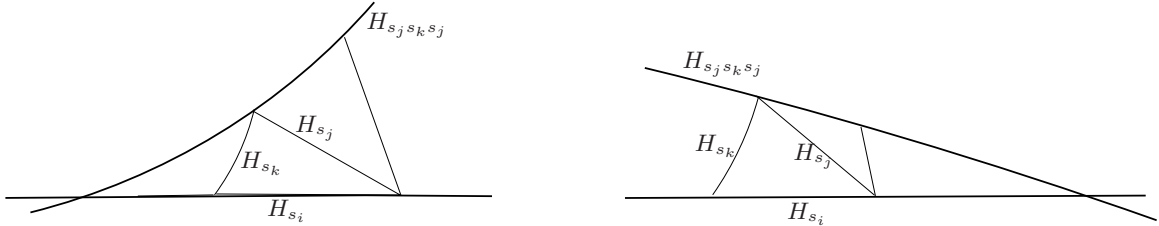,width=0.95\linewidth}
\caption{Towards the proof of Lemma~\ref{2copies}, Case~2}  
\label{i-jkj}
\end{center}
\end{figure}

\begin{lemma}
\label{2copies-long} Suppose $s_i,s_j, s_k$ are distinct elements of $S$ such that 
$m_{ij}, m_{ik}, m_{kj}$ are odd. 
Then  $s_js_ks_j$ and $s_ks_is_ks_is_k$ generate an infinite dihedral group.

\end{lemma}

\begin{proof}
First, suppose that $m_{ik}=3$. Then $s_ks_is_ks_is_k=s_i$ and the statement follows from Lemma~\ref{2copies}.

Now, suppose $m_{ik}\ne 3$. If we assume that the mirrors of the corresponding reflections intersect each other, then, as in the proof of Lemma~\ref{2copies}, we can construct a hyperbolic triangle tessellated by smaller Coxeter triangles, which leads to a contradiction due to results of~\cite{F}.  
\end{proof}

\begin{lemma} \label{geodesic-gallery-lemma-2}
Let $(W, S)$ be a Coxeter system and $P=\bigcap\limits_{i=1}^{n} H^+_{w_i}$ a polytope in $\Sigma(W,S)$. If $v_1$ and  $v_k$ are elements of $W$ with $D(v_1), D(v_k) \subset P$, and $(D(v_1), D(v_2), \hdots, D(v_k))$ is a geodesic gallery, then $D(v_i) \subset P$ for all $1 \leq i \leq k$.
\end{lemma}

\begin{proof} Suppose that $1 < i < k$, $D(v_i) \not \subset P$. Then there exists $1 \leq j \leq n$ such that $D(v_i)\not\subset H_{w_j}^+$. Hence $(D(v_1), \hdots, D(v_i))$ and $(D(v_i), \hdots, D(v_k))$ cross the wall $H_{w_j}$, and by Lemma~\ref{geodesic-gallery-lemma} the gallery $(D(v_1), D(v_2), \hdots, D(v_k))$ is not geodesic, which contradicts the assumptions.
 \end{proof}

\noindent
{Note also that in view of Lemma~\ref{geodesic-gallery-lemma}, Lemma~\ref{geodesic-gallery-lemma-2} follows immediately from~\cite[Proposition~2.6]{R}. 
}

The main technical tool will consist of the following two lemmas.

\begin{lemma}
\label{2ways}
Let $D(w_0)$ be a chamber of $\Sigma(W,S)$. Suppose that there are two geodesic galleries $(D(w_0)$, $D(w_1), \dots, D(w_k))$ and $(D(w_0),D(w'_1),\dots,D(w'_{k'}))$, such that $l(w_i)>l(w_j)$ and $l(w'_i)>l(w'_j)$ for $i<j$. Let $H$ be a wall intersecting the first gallery but not the second, and $H'$ be a wall intersecting the second gallery but not the first one. Then $H\cap H'\ne\emptyset$.  
\end{lemma}

\begin{proof}
We can extend the geodesic galleries $(D(w_0)$, $D(w_1), \dots, D(w_k))$ and $(D(w_0),D(w'_1),\dots,D(w'_{k'}))$ to geodesic galleries $\Gamma=(D(w_0)$, $D(w_1), \dots, D(w_k), D(w_{k+1}), \hdots, D(1))$ and $\Gamma'=(D(w_0),D(w'_1),\dots,$ $D(w'_{k'}),D(w'_{k+1}), \hdots, D(1))$. Denote by $H^+$ and $H'^+$ the halfspaces defined by $H$ and $H'$ that contain $D(w_0)$. As $\Gamma$ crosses the wall $H$ and $\Gamma'$ crosses $H'$, the identity element $D(1)$ has to lie in the complementary halfspaces, i.e. in $H^- \cap H'^-$, see Lemma \ref{geodesic-gallery-lemma}. This requires in particular that $H^- \cap H'^- \neq \emptyset \neq H^+ \cap H'^+$, which implies that the intersection between $H$ and $H'$ is nonempty.
\end{proof}

\begin{cor}
\label{inf neighb}
Let $w\in W$, $s_i,s_j\in S$. If $l(w)=l(ws_i)+1=l(ws_j)+1$, then $m_{ij}$ is finite. 

\end{cor}

\begin{proof}
Applying the previous lemma to the geodesic galleries $(D(w), D(ws_i))$ and $(D(w), D(ws_j))$ shows that $H_{ws_iw^{-1}} \cap H_{ws_jw^{-1}} \neq \emptyset$, and the order $m_{ij}$ of $s_is_j$ is finite.
\end{proof}

\begin{lemma}
\label{lemma7} Suppose that $\Sigma(W,S)$ has dimension two.
	Let $s_i, s_j \in S$, $s_i \neq s_j$, and $w \in I$ such that $l(w)=l(ws_i)+1=l(ws_j)+1$. Then for every integer $k$ the elements $w(s_is_j)^k$ and $w(s_is_j)^ks_i$ are in $I$. 
	
	Moreover, $l(w(s_is_j)^k)=l(w(s_js_i)^k)=l(w)-2k$ for $k\leq m_{ij}/2$, and 
$l(w(s_is_j)^ks_i)=l(w(s_js_i)^ks_j)=l(w)-2k-1$ for $k\leq (m_{ij}-1)/2$.

\end{lemma}
\begin{proof}

We note first that $m_{ij}<\infty$ by Corollary~\ref{inf neighb}.

Denote $l(w)$ by $m$. By Corollary~\ref{ending-element-corollary}, there exists a reduced expression $E_1$ for $w$ that ends in $s_i$ and another reduced expression $E_2$ that ends in $s_j$. Theorem~\ref{M-operation} implies that $E_1$ can be transformed into $E_2$ by repeatedly applying $M$-operations of second type (given in Theorem \ref{M-operation}). Since $\Sigma$ has dimension two, every subgroup generated by at least three different elements of $S$ is infinite. Hence, using Lemma~\ref{end-of-reduced-expression}, we see that this sequence of reduced expressions starting with $E_1$, ending with $E_2$ and obtained by repeatedly applying $M$-operations of the second type contains only expressions ending in $s_i$ or $s_j$. As $E_1$ ends in $s_i$ and $E_2$ ends in $s_j$, there exists a reduced expression $\wt E= \widetilde s_1 \widetilde s_2 \cdots \widetilde s_m$ for $w$ in the above sequence that ends in $s_i$ and that gets transformed into an expression ending in $s_j$ by applying $M$-operation of second type once. This is only possible if $\widetilde E$ ends in $(s_js_i)^{m_{ij}/2}$ for even $m_{ij}$ or $(s_is_j)^{(m_{ij}-1)/2}s_i$ if $m_{ij}$ is odd.

Now consider $\widetilde E$ as a geodesic gallery joining $D(1)$ with $D(w)$. In view of Lemma~\ref{geodesic-gallery-lemma-2}, we see that $w(s_is_j)^k\in I$ and $w(s_is_j)^ks_i\in I$ for all $k\in \Z$.  Since the length of two neighbors differs exactly by $1$, the elements $w$, $w(s_is_j)^k$ have length $m-2k$ for $k\leq m_{ij}/2$, and $w(s_is_j)^ks_i$ has length $m-2k-1$ for $k\leq (m_{ij}-1)/2$.
\end{proof}

\begin{remark}
\label{less}
Similar to the last paragraph of the proof above, we make the following observation, { which} we will { use} throughout the paper. If $w$ is a reduced expression, and $D(w)\in P$ (i.e., $w\in I$), then for any $s_i\in S$ such that $l(ws_i)=l(w)-1$ the chamber $D(ws_i)$ is contained in $P$.  

\end{remark}

\begin{remark}
\label{highdim}
{Lemma~\ref{lemma7} still holds if we drop the assumption on $\Sigma(W,S)$ to have dimension two. This can be proved using~\cite[Theorem~2.9 and Lemma~2.10]{R}.}

\end{remark}

\pagebreak

\begin{lemma}
\label{inf}
Suppose that $w,ws_i\in I$ and $ws_k,ws_is_k\notin I$. Then $m_{ik}=\infty$.

\end{lemma}

\begin{proof}
The lemma follows immediately from the assumption that $W$ is odd-angled. Indeed, assume that $m_{ik}\ne\infty$. Then the walls $H_{ws_kw^{-1}}$ and $H_{ws_is_ks_iw^{-1}}$ form an angle equal to $2\pi/m_{ik}$, which cannot be an angle of $P$ as $P$ is a Coxeter polytope (see Fig.~\ref{inf-fig}).   
\end{proof}

\begin{figure}[!h]
\begin{center}
\psfrag{wsk}{\small $ws_k$}
\psfrag{w}{\small $w$}
\psfrag{wsi}{\small $ws_i$}
\psfrag{wsik}{\small $ws_is_k$}
\psfrag{1}{\scriptsize $H_{ws_kw^{-1}}$}
\psfrag{2}{\scriptsize $H_{ws_is_ks_iw^{-1}}$}
\epsfig{file=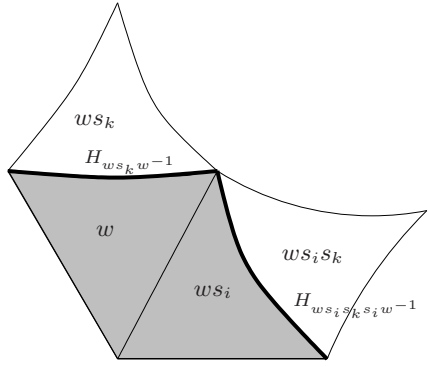,width=0.35\linewidth}
\caption{Towards the proof of Lemma~\ref{inf}}  
\label{inf-fig}
\end{center}
\end{figure}

\begin{lemma}
\label{lemma b}
The lengths of the elements $w,ws_i,ws_is_j,ws_is_js_k$ cannot be $L-1,L,L-1,L-2$ for any $L>2$ and distinct $i,j,k$.

\end{lemma}

\begin{proof}
Suppose \ the \ contrary. Then \ the \ geodesic \ galleries \ $(D(ws_i)$, $D(w))$ \ and \ $(D(ws_i)$, $D(ws_is_j)$, $D(ws_is_js_k))$ together with walls $H_1$ separating $D(ws_i)$ and $D(w)$ and $H_2$ separating $D(ws_is_j)$ and $D(ws_is_js_k)$ satisfy the assumptions of  Lemma~\ref{2ways}. Therefore, these two walls must intersect. However, this contradicts Lemma~\ref{2copies}, see Fig.~\ref{f lemma ab}.a.
\end{proof}

\begin{figure}[!h]
\begin{center}
\psfrag{f}{\scriptsize $f$} 
\psfrag{fsk}{\scriptsize $fs_k$} 
\psfrag{fsi}{\scriptsize $fs_i$} 
\psfrag{fsij}{\scriptsize $fs_is_j$} 
\psfrag{fsijk}{\scriptsize $fs_is_js_k$} 
\psfrag{fsiji}{\scriptsize $fs_is_js_i$} 
\psfrag{w}{\scriptsize $w$} 
\psfrag{wsk}{\scriptsize $ws_k$} 
\psfrag{wsi}{\scriptsize $ws_i$} 
\psfrag{wsij}{\scriptsize $ws_is_j$} 
\psfrag{wsijk}{\scriptsize $ws_is_js_k$} 
\psfrag{wsiji}{\scriptsize $ws_is_js_i$} 
\psfrag{L}{\small $L$} 
\psfrag{L-1}{\small $L-1$} 
\psfrag{L-2}{\small $L-2$} 
\psfrag{M}{\small $M$} 
\psfrag{M-1}{\small $M-1$} 
\psfrag{M-2}{\small $M-2$}
\psfrag{a}{\small a)}
\psfrag{b}{\small b)} 
\epsfig{file=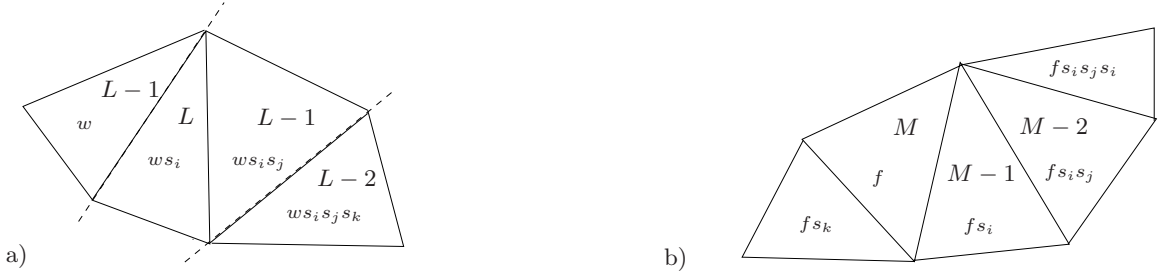,width=0.95\linewidth}
\caption{Towards the proofs of Lemmas~\ref{lemma b} and~\ref{lemma a}}  
\label{f lemma ab}
\end{center}
\end{figure}

The following statement, which we will use throughout the section, follows immediately from the definition of a Coxeter polytope.
\begin{lemma}
\label{divide}
Let $w\in W$, and let $Q=H_{w_1}\cap H_{w_2}\cap P$ be a codimension $2$ face of $P$, where $w_l=w(s_is_j)^{k_l}s_iw^{-1}$, {$l=1,2$ and $k_1,k_2<m_{ij}$ are some non-negative integers. Then the number of copies of the fundamental chamber $D(1)$ in $P$ containing $Q$ divides $m_{ij}$.}  
\end{lemma}

\pagebreak

We now consider an element $f$ of $I$ of maximal length. We will denote its length by $M$.

\begin{lemma}
\label{5angle}
Let $f\in I$ be an element of maximal length $M$. Assume that $fs_i\in I$ and  $m_{ij}\ne \infty$. Then $fs_is_j\in I$.
Moreover, if $3\,\not{|}\,m_{ij}$ then  $l(fs_is_j)=l(fs_i)-1=M-2$, and  if $3\,\not{|}\,m_{ij}$, $5\,\not{|}\,m_{ij}$ then $l(fs_is_js_i)=M-3$. 

\end{lemma}

\begin{proof}
Since $fs_i\in I$, we have $l(fs_i)=M-1$. First, assume that $fs_j\in I$. Then $l(fs_j)=M-1$ and by Lemma~\ref{lemma7}
$f(s_is_j)^k\in I$ and $f(s_is_j)^ks_i\in I$ for all $k\in \Z$, and $l(fs_is_j)=M-2$, $l(fs_is_js_i)=M-3$.  

From now on we assume that $fs_j\notin I$. 
By Lemma~\ref{inf}, $fs_is_j\in I$, which proves the first part of the lemma. 

Suppose that $3\,\not{|}\,m_{ij}$.  Then{, by Lemma~\ref{divide},} $fs_is_js_i \in I$. If $l(fs_is_j)=M$,
then $l(fs_is_js_i)=M-1$ and, by Lemma~\ref{lemma7}, all elements $fs_is_j(s_is_j)^k$ and $fs_is_j(s_is_j)^ks_i$ are contained in $I$ for all $k\in \Z$, which contradicts the assumption that $fs_j\notin I$. This implies that $l(fs_is_j)=M-2$.

Finally, suppose in addition that $5\,\not{|}\,m_{ij}$. We have already proved that $fs_is_js_i \in I$, so we are left to show that $l(fs_is_js_i)=M-3$. If we assume the contrary (i.e., $l(fs_is_js_i)=M-1$), then either $l(f(s_is_j)^2)=M-2$ and Lemma~\ref{lemma7} leads to a contradiction with $fs_j\notin I$, or $l(f(s_is_j)^2)=M$. In the latter case since $5\,\not{|}\,m_{ij}$, we see that $f(s_is_j)^2s_i\in I$, so $l(f(s_is_j)^2s_i)=M-1$, and we again apply Lemma~\ref{lemma7} to obtain a contradiction  with $fs_j\notin I$.
\end{proof}

\begin{lemma}
\label{lemma a}
Let $f\in I$ and $M=l(f)=l(fs_i)+1=l(fs_is_j)+2$. Then 

(a) $l(fs_k)=M+1$, i.e. $fs_k\notin I$ for  all $k$ distinct from $i$ and $j$;

(b) if $l(fs_is_js_i)=M-1$, then $l(fs_j)=M+1$ and $fs_j\notin I$. 

\end{lemma}

\begin{proof}
First, $l(fs_k)=M-1$ is impossible by Lemma~\ref{lemma b} (applied to the elements $fs_k,f,fs_i,fs_is_j$ 
of lengths $M-1,M,M-1,M-2$, see Fig.~\ref{f lemma ab}.b), which proves the first part of the lemma. 
Now, suppose that $l(fs_j)=M-1$. Then, by Lemma~\ref{lemma7} (applied to the elements $fs_j,f,fs_i$ of lengths 
$M-1,M,M-1$), we see that  $l(fs_is_js_i)=M-3$. Hence,   $l(fs_is_js_i)=M-1$ implies  $l(fs_j)=M+1$, i.e. $fs_j\notin I$.
\end{proof}

\subsection{Rank $3$ groups}
\label{rank 3}

In Lemma~\ref{triangles} we discuss rank $3$ odd-angled Coxeter groups $W$ with 3 defining relations 
$(s_1s_2)^{m_{12}}=(s_1s_3)^{m_{13}}=(s_2s_3)^{m_{23}}=1$, $m_{ij}\ne \infty$. The case of two defining relations (i.e., one of $m_{ij}$ is equal to $\infty$) is considered in Lemma~\ref{55inf}. Rank $3$ Coxeter groups with at most one $m_{ij}\in\Z$ have disconnected divisibility diagrams. We consider them in Section~\ref{min}.

As was defined above, let $V\subset W$ be a  finite index reflection subgroup, $P$ be its fundamental domain containing $D(1)$, and let be $I$ the set of elements $w$ of $W$ with $D(w)\in P$. We denote by $f$ the element of $I$ of maximal length $M$. We will also assume that $V\subset W$ is a reflection subgroup of smallest possible index, hence $P$ does not (strictly) contain any Coxeter polytope except chambers of $\Sigma(W,S)$.

\begin{lemma} 
\label{triangles}
Let $(m_{ij},m_{ik},m_{jk})$ be distinct from $(3k_1,3k_2,k_3)$ and $(5k_1,5k_2,3k_3)$. 
Then $W$ has no finite index reflection subgroups.

\end{lemma}

\begin{proof}
Suppose that $V\subset W$ is a finite index reflection subgroup. There exists $s_i$ (say, $s_1$) such that $fs_1\in I$, $l(fs_1)=M-1$. By Lemma~\ref{5angle}, there is $s_j$ (say, $s_2$) such that $fs_1s_2\in I$, $l(fs_1s_2)=M-2$. 
By Lemma~\ref{lemma a} $fs_3\notin I$ (see Fig.~\ref{fig1_2}). 
Consider the element $fs_1s_3\in I$. Its length is either $M$ or $M-2$.


\begin{figure}[!h]
\begin{center}
\psfrag{a}{a)}
\psfrag{b}{b)}
\psfrag{c}{c)}
\psfrag{f}{\scriptsize $f$} 
\psfrag{fs1}{\scriptsize $fs_1$} 
\psfrag{fs12}{\scriptsize $fs_1s_2$} 
\psfrag{fs13}{\scriptsize $fs_1s_3$} 
\psfrag{fs3}{\scriptsize $fs_3$} 
\psfrag{M}{\tiny $M$} 
\psfrag{M-1}{\tiny $M-1$} 
\psfrag{M-2}{\tiny $M-2$} 
\psfrag{?}{\scriptsize $?$} 
\epsfig{file=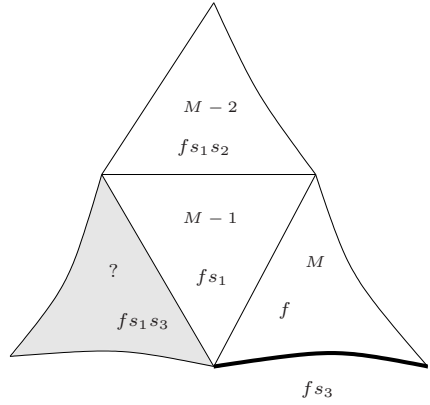,width=0.35\linewidth}
\caption{Towards the proof of Lemma~\ref{triangles}}  
\label{fig1_2}
\end{center}
\end{figure}

\noindent
{\bf Case 1:} $l(fs_1s_3)=M$.\\       
According to Lemma~\ref{lemma7}, $fs_1s_3s_1$ does not belong to $I$, {so Lemma~\ref{divide}} implies that $m_{13}$ is divisible by $3$. 
Therefore, $3$ divides neither $m_{12}$ nor $m_{23}$. Thus, either $fs_2$ or $fs_1s_2s_1$ belongs to $I$. \\

\noindent
{\bf Case 1.1:} $fs_2\in I$ (see Fig.~\ref{fig8_9_10}).\\ 
Clearly, $l(fs_2)=M-1$. So, by Lemma~\ref{5angle},  $fs_2s_3\in I$ and 
$l(fs_2s_3)=M-2$, which contradicts Lemma~\ref{lemma b} applied to elements $fs_2s_3,fs_2,f,fs_1$ of lengths $M-2,M-1,M,M-1$ respectively.\\

\begin{figure}[!h]
\begin{center}
\psfrag{a}{a)}
\psfrag{b}{b)}
\psfrag{c}{c)}
\psfrag{?}{$?$} 
\psfrag{f}{\tiny $f$} 
\psfrag{M-1}{\tiny $M-1$} 
\psfrag{M-2}{\tiny $M-2$} 
\psfrag{M-3}{\tiny $M-3$} 
\psfrag{M}{\tiny $M$} 
\psfrag{fs1}{\tiny $fs_1$} 
\psfrag{fs12}{\tiny $fs_1s_2$} 
\psfrag{fs123}{\tiny $fs_1s_2s_3$} 
\psfrag{fs132}{\tiny $fs_1s_3s_2$} 
\psfrag{fs121}{\tiny $fs_1s_2s_1$} 
\psfrag{fs13}{\tiny $fs_1s_3$} 
\psfrag{fs2}{\tiny $fs_2$} 
\psfrag{fs23}{\tiny $fs_2s_3$} 
\epsfig{file=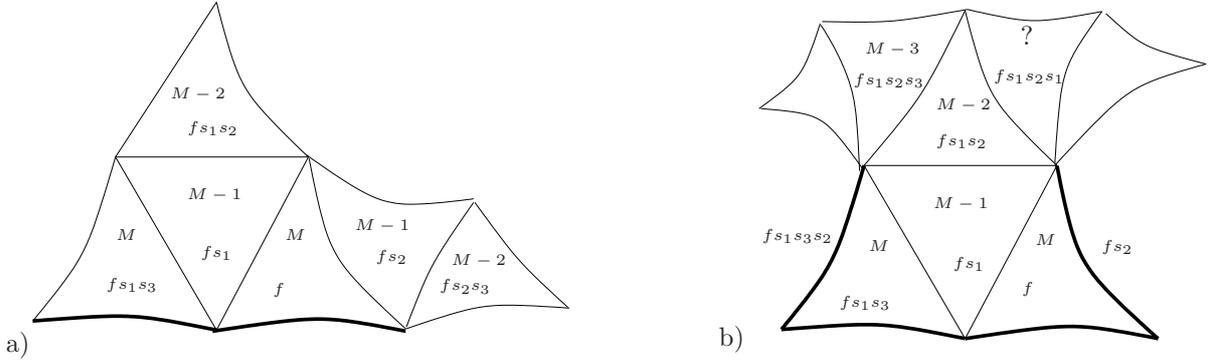,width=0.99\linewidth}
\caption{Proof of Lemma~\ref{triangles}: (a) Case 1.1, (b) Case 1.2 }  
\label{fig8_9_10}
\end{center}
\end{figure}

\pagebreak
\noindent
{\bf Case 1.2:} $fs_2\notin I$.\\ 
By symmetry, we may also assume that $fs_1s_3s_2\notin I$  (see Fig.~\ref{fig8_9_10}b). Then all the neighbors of 
$fs_1s_2$ belong to $I$. One of the neighbors of $fs_1s_2$ (say, $fs_1s_2s_3$) has length $M-3$. 
Consider the second neighbor, $fs_1s_2s_1$. Its length is either $M-1$ or $M-3$.

\begin{figure}[!h]
\begin{center}
\psfrag{a}{a)}
\psfrag{b}{b)}
\psfrag{c}{c)}
\psfrag{f}{\tiny $f$} 
\psfrag{M-1}{\tiny $M-1$} 
\psfrag{M-2}{\tiny $M-2$} 
\psfrag{M}{\tiny $M$} 
\psfrag{M-3}{\tiny $M-3$} 
\psfrag{fs1}{\tiny $fs_1$} 
\psfrag{fs2}{\tiny $fs_2$} 
\psfrag{fs13}{\tiny $fs_1s_3$} 
\psfrag{fs132}{\tiny $fs_1s_3s_2$}
\psfrag{fs12}{\tiny $fs_1s_2$} 
\psfrag{fs123}{\tiny $fs_1s_2s_3$}  
\psfrag{fs121}{\tiny $fs_1s_2s_1$}  
\psfrag{fs1212}{\tiny $fs_1s_2s_1s_2$}  
\psfrag{fs1213}{\tiny $fs_1s_2s_1s_3$}  
\psfrag{fs12132}{\tiny $fs_1s_2s_1s_3s_2$}  
\psfrag{?}{\tiny $?$} 
\psfrag{6}{\tiny $M\!\!-\!\!1$} 
\psfrag{8}{\tiny $M\!\!-\!\!2$} 
\psfrag{9}{\tiny $M\!\!-\!\!3$} 
\epsfig{file=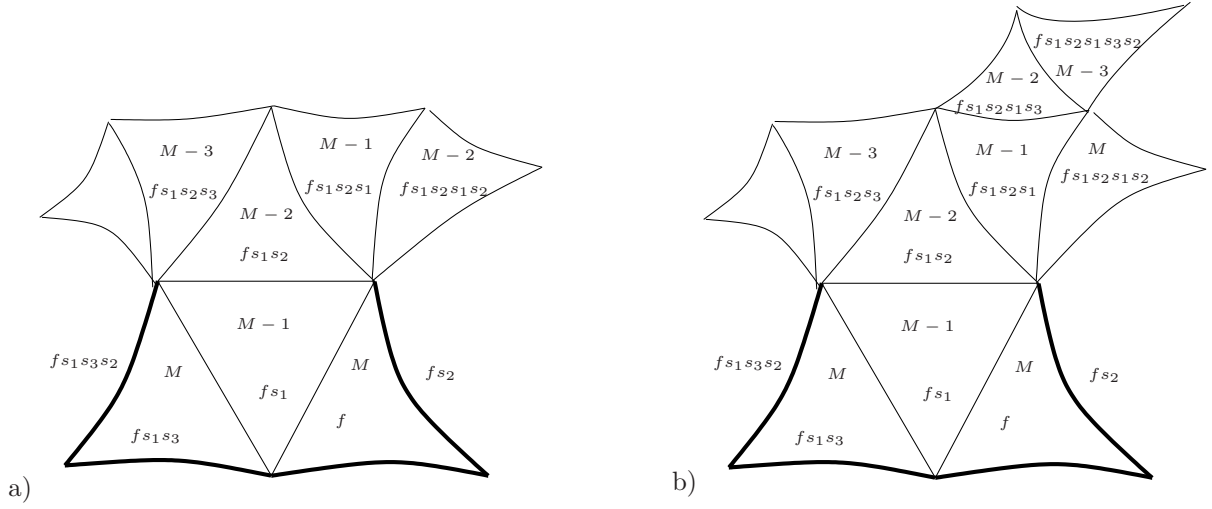,width=0.99\linewidth}
\caption{Proof of Lemma~\ref{triangles}, Case 1.2, continuation}  
\label{fig11_12}
\end{center}
\end{figure}

Assume first that  $l(fs_1s_2s_1)=M-1$, and consider $fs_1s_2s_1s_2$. Its length is either $M-2$ or $M$.
 
If $l(fs_1s_2s_1s_2)=M-2$  (see Fig.~\ref{fig11_12}a), then by Lemma~\ref{lemma7} $fs_2\in I$, which contradicts the assumption of Case~1.2. Therefore, $l(fs_1s_2s_1s_2)=M$  (see Fig.~\ref{fig11_12}b). Moreover $l(fs_1s_2s_1s_2s_1)=M+1$, and $fs_1s_2s_1s_2s_1\notin I$, which implies that $m_{12}$ is a multiple of $5$ {(again, we use Lemma~\ref{divide})}. In view of the assumptions of the lemma, this means that $m_{23}$ is neither a multiple of $3$ nor a multiple of $5$, so we can apply Lemma~\ref{5angle} to obtain that $l(fs_1s_2s_1s_3)=M-2$  and $l(fs_1s_2s_1s_3s_2)=M-3$. 
This contradicts Lemma~\ref{lemma b} (applied to the elements $f_Ms_1s_2,f_Ms_1s_2s_1,f_Ms_1s_2s_1s_3,f_Ms_1s_2s_1s_3s_2$
of lengths $M-2,M-1,M-2,M-3$).

\begin{figure}[!h]
\begin{center}
\psfrag{a}{a)}
\psfrag{b}{b)}
\psfrag{f}{\tiny $f$} 
\psfrag{M-1}{\tiny $M-1$} 
\psfrag{M-2}{\tiny $M-2$} 
\psfrag{M}{\tiny $M$} 
\psfrag{M-3}{\tiny $M-3$} 
\psfrag{fs1}{\tiny $fs_1$} 
\psfrag{fs2}{\tiny $fs_2$} 
\psfrag{fs13}{\tiny $fs_1s_3$} 
\psfrag{fs132}{\tiny $fs_1s_3s_2$}
\psfrag{fs12}{\tiny $fs_1s_2$} 
\psfrag{fs123}{\tiny $fs_1s_2s_3$}  
\psfrag{fs1232}{\tiny $fs_1s_2s_3s_2$}  
\psfrag{fs12323}{\tiny $fs_1s_2s_3s_2s_3$}  
\psfrag{fs123232}{\tiny $fs_1s_2s_3s_2s_3s_2$}  
\psfrag{fs123231}{\tiny $fs_1s_2s_3s_2s_3s_1$}  
\psfrag{fs121}{\tiny $fs_1s_2s_1$}  
\psfrag{fs1212}{\tiny $fs_1s_2s_1s_2$}  
\psfrag{fs1213}{\tiny $fs_1s_2s_1s_3$}  
\psfrag{fs12132}{\tiny $fs_1s_2s_1s_3s_2$}  
\psfrag{?}{\tiny $?$} 
\epsfig{file=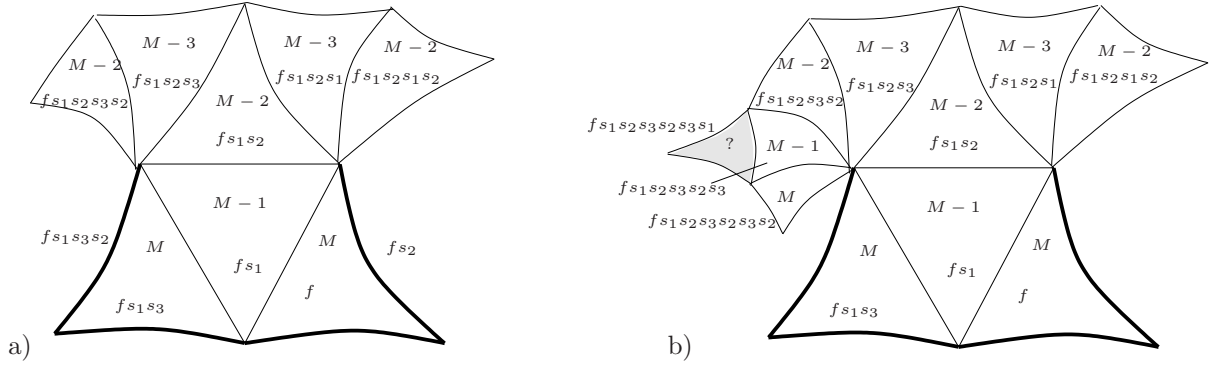,width=0.99\linewidth}
\caption{Case 1.2, continuation}  
\label{fig13_14_15}
\end{center}
\end{figure}

Thus, we may assume that $l(fs_1s_2s_1)=l(fs_1s_2s_3)=M-3$, see Fig.~\ref{fig13_14_15}a. Consider the neighbors of the above elements, $fs_1s_2s_1s_2$ and $fs_1s_2s_3s_2$, both in $I$. 
By Lemma~\ref{lemma b} (applied to the elements $fs_1s_2s_3,fs_1s_2,fs_1s_2s_1,fs_1s_2s_1s_2$) we have $l(fs_1s_2s_1s_2)=M-2$.
Similarly, $l(fs_1s_2s_3s_2)=M-2$. 

By the assumptions of the lemma, we may assume that either $m_{12}$ or $m_{23}$  (say, $m_{23}$) is not a multiple of $5$. Therefore, we may assume that $fs_1s_2s_3s_2s_3$ and $fs_1s_2s_3s_2s_3s_2$ belong to $I$ (see Fig.~\ref{fig13_14_15}b). 
By Lemma~\ref{lemma7}, $l(fs_1s_2s_3s_2s_3)=M-1$, $l(fs_1s_2s_3s_2s_3s_2)=M$. Consider $fs_1s_2s_3s_2s_3s_1$ , its length is either $M$ or $M-2$.

The latter case is impossible by Lemma~\ref{lemma b} (applied to the elements 
$fs_1s_2s_3,fs_1s_2s_3s_2,fs_1s_2s_3s_2s_3$, $fs_1s_2s_3s_2s_3s_1$ of lengths $M-3,M-2,M-1,M-2$).
In the former case, $fs_1s_2s_3s_2s_3s_1\in I$, and, reasoning as in the beginning of consideration of Case~1, we obtain that $m_{12}$ is divisible by $3$, which contradicts assumptions of the lemma.\\

\noindent
{\bf Case 2:} $l(fs_1s_3)=M-2$ { (see Fig.~\ref{fig5_6_7}a)}.\\
Similar to $fs_3$, the element $fs_2$ does not belong to $I$ either. Since at most one of the $m_{ij}$ is divisible by $3$, { and the picture is symmetric with respect to the interchange of the indices $2$ and $3$, we may assume that $m_{13}$ and at least one of $m_{12}$ and $m_{23}$ are not divisible by $3$.}  
Then $fs_1s_3s_1\in I$, its length is either $M-1$ or $M-3$. 
The latter is impossible by Lemma~\ref{lemma b} (applied to the elements $fs_1s_2,fs_1,fs_1s_3,fs_1s_3s_1$ of lengths
$M-2,M-1,M-2,M-3$).
So, $l(fs_1s_3s_1)=M-1$.
Consider $fs_1s_3s_1s_3\in I$, its length is either $M-2$ or $M$. In the former case, Lemma~\ref{lemma7} implies that $fs_3\in I$, which does not hold. Hence, $l(fs_1s_3s_1s_3)=M$  { (see Fig.~\ref{fig5_6_7}b)}. In particular, $fs_1s_3s_1s_3s_1\notin I$, and $m_{13}$ is a multiple of $5$.

\begin{figure}[!h]
\begin{center}
\psfrag{a}{a)}
\psfrag{b}{b)}
\psfrag{c}{c)}
\psfrag{f}{\tiny $f$} 
\psfrag{M-1}{\tiny $M-1$} 
\psfrag{M-2}{\tiny $M-2$} 
\psfrag{M}{\tiny $M$} 
\psfrag{M-3}{\tiny $M-3$} 
\psfrag{fs1}{\tiny $fs_1$} 
\psfrag{fs2}{\tiny $fs_2$} 
\psfrag{fs13}{\tiny $fs_1s_3$} 
\psfrag{fs132}{\tiny $fs_1s_3s_2$}
\psfrag{fs12}{\tiny $fs_1s_2$} 
\psfrag{fs123}{\tiny $fs_1s_2s_3$}  
\psfrag{fs1232}{\tiny $fs_1s_2s_3s_2$}  
\psfrag{fs12323}{\tiny $fs_1s_2s_3s_2s_3$}  
\psfrag{fs123232}{\tiny $fs_1s_2s_3s_2s_3s_2$}  
\psfrag{fs123231}{\tiny $fs_1s_2s_3s_2s_3s_1$}  
\psfrag{fs121}{\tiny $fs_1s_2s_1$}  
\psfrag{fs131}{\tiny $fs_1s_3s_1$}  
\psfrag{fs1313}{\tiny $fs_1s_3s_1s_3$}  
\psfrag{fs1212}{\tiny $fs_1s_2s_1s_2$}  
\psfrag{fs1312}{\tiny $fs_1s_3s_1s_2$}  
\psfrag{fs12132}{\tiny $fs_1s_2s_1s_3s_2$}  
\psfrag{fs131232}{\tiny $fs_1s_3s_1s_2s_3s_2$}  
\psfrag{fs13123}{\tiny $fs_1s_3s_1s_2s_3$}  
\psfrag{fs131231}{\tiny $fs_1s_3s_1s_2s_3s_1$}  
\psfrag{?}{\tiny $?$} 
\epsfig{file=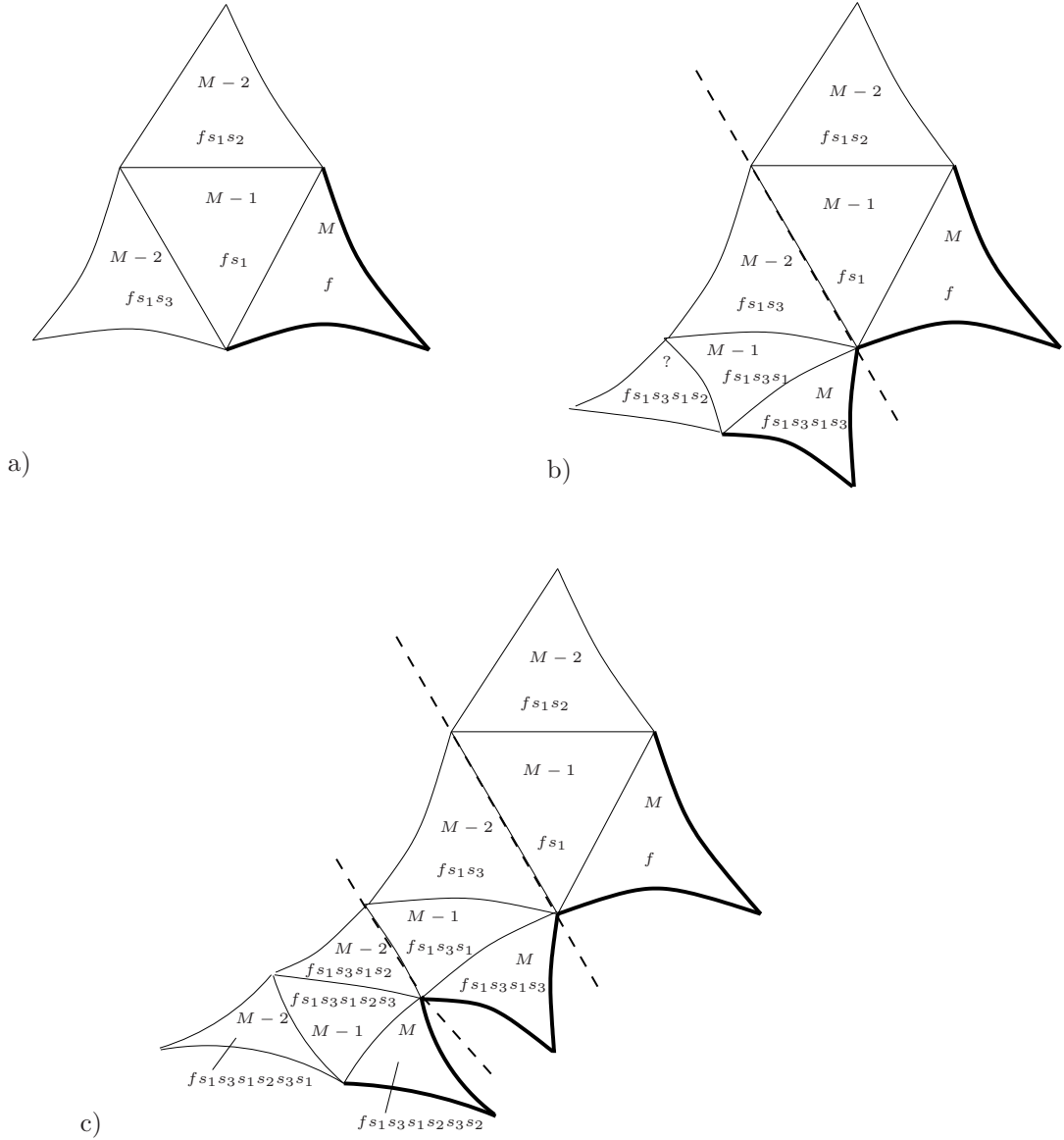,width=0.89\linewidth}
\caption{Proof of Lemma~\ref{triangles}, Case 2.}  
\label{fig5_6_7}
\end{center}
\end{figure}

Now consider $fs_1s_3s_1s_2\in I$ (see Fig.~\ref{fig5_6_7}b). If its length is $M$, then we are in Case~1, so we come to a contradiction. Therefore, $l(fs_1s_3s_1s_2)=M-2$. { Recall that, by our assumption,  at least one of $m_{12}$ and $m_{23}$ is not divisible by $3$. Consider two cases.\\

\noindent
{\bf Case 2.1:} $m_{23}$ is not divisible by $3$.\\
Repeating the reasonings as above (i.e., replacing the elements $f,fs_1,fs_1s_2,fs_1s_3$ by $fs_1s_3s_1s_3,fs_1s_3s_1$, $fs_1s_3,fs_1s_3s_1s_2$), we obtain that $m_{23}$ is divisible by $5$ (see Fig.~\ref{fig5_6_7}c). Therefore, by the assumptions of the lemma, $m_{12}$ is not divisible by $3$ either, and we can continue constructing elements of $I$ in the same way (now replacing the elements $f,fs_1,fs_1s_2,fs_1s_3$ by $fs_1s_3s_1s_2s_3s_2,fs_1s_3s_1s_2s_3,fs_1s_3s_1s_2,fs_1s_3s_1s_2s_3s_1)$. 

According to Lemma~\ref{2copies}, the wall separating chambers $D(fs_1)$ and $D(fs_1s_3)$ does not intersect the wall separating chambers $D(fs_1s_3s_1)$ and $D(fs_1s_3s_1s_2)$. Similarly, these walls do not intersect the wall separating chambers $D(fs_1s_3s_1s_2s_3)$ and $D(fs_1s_3s_1s_2s_3s_1)$ (see Fig.~\ref{fig5_6_7}c). This implies that iterating the procedure above we can construct arbitrary many elements of $I$, which contradicts the finiteness of the index of $V$.\\

\noindent
{\bf Case 2.2:} $m_{12}$ is not divisible by $3$.\\
Recall that from the assumption that $m_{13}$ is not divisible by $3$ we have deduced that $m_{13}$ is a multiple of $5$. Since Figure~\ref{fig5_6_7}a is symmetric with respect to the interchange of $s_2$ and $s_3$, the assumption that $m_{12}$ is not divisible by $3$ implies that $m_{12}$ is also divisible by $5$. Then, by the assumptions of the lemma, $m_{23}$ is not divisible by $3$, and we are in the assumptions of Case~2.1.}        
\end{proof}

\begin{lemma}
\label{55inf}
Let $(W,S)=\langle s_1,s_2,s_3\ |\ (s_1s_2)^{m_{12}}=(s_2s_3)^{m_{23}}=1 \rangle$ be a Coxeter system such that ${m_{12}}{m_{23}}$ is not divisible by $2$ and $3$. Then $W$ has no finite index reflection subgroup.

\end{lemma}

\begin{proof}
Suppose the contrary, i.e. let $V$ be a  finite index reflection subgroup of smallest index. The complex $\Sigma(W,S)$ can be identified with the hyperbolic plane $\H^2$, the chamber $D(1)$ (or $D(w)$ for any $w\in W$) with { a} fundamental triangle of $W$, and  $P$ with { a} fundamental polygon of $V$). First, we show that no copy of $D(1)$ in the tessellation of $P$ has a vertex in the interior of $P$. 

Assume that there is  a  vertex $X$ of some copy of $D(1)$ (denote this copy by $F$) contained in the interior of $P$. Let $Y$ be the other vertex of $F$ with non-zero angle, and let $Z$ be the third vertex of $F$ 
(it lies at the boundary of $\H^2$ { since $m_{13}=\infty$ and the angle at $Z$ is $0$}).  Consider the angle $\angle XYZ$ formed by the rays $YX$ and $YZ$.
It is clear that $P'=P\cap \angle XYZ$ is a Coxeter polygon (see Fig.~\ref{f55inf}): 
all its angles are either angles of $P$,
or the angle $Y$ of $F$, or the zero angles at vertices at infinity (since $m_{12}$ and $m_{23}$ are odd). Furthermore, $P'$ contains more than one 
fundamental triangle of $W$.
So, $P'$ is a fundamental domain of some  subgroup of $W$ with index smaller than the index of $V$,
which contradicts the assumptions.

\begin{figure}[!h]
\begin{center}
\psfrag{a}{a)}
\psfrag{b}{b)}
\psfrag{c}{c)}
\psfrag{F}{\large $F$} 
\psfrag{X}{$X$} 
\psfrag{X'}{ $X'$} 
\psfrag{Y}{ $Y$} 
\psfrag{Z}{ $Z$} 
\epsfig{file=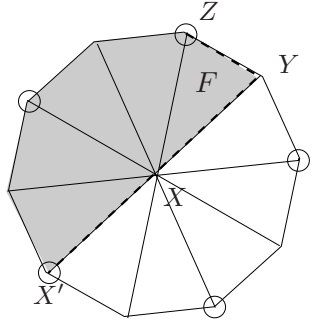,width=0.25\linewidth}
\caption{Proof of Lemma~\ref{55inf}: no vertices in the interior of $P$. The circles show the vertices at infinity.}  
\label{f55inf}
\end{center}
\end{figure}

Therefore, every vertex of any fundamental chamber in the tessellation of $P$ belongs to the boundary of $P$.

Now, using the notation as above, we will show that $fs_2\notin I$. Indeed, suppose $fs_2\in I$. Then $l(fs_2)=M-1$. 
By Lemma~\ref{5angle}, $fs_2s_1\in I$, and $l(fs_2s_1)=M-2$. Similarly, $fs_2s_3\in I$, and $l(fs_2s_3)=M-2$
(see Fig.\ref{55inf2}.a), which contradicts Corollary~\ref{inf neighb}.

\begin{figure}[!h]
\begin{center}
\psfrag{f}{\tiny $f$} 
\psfrag{fs3}{\tiny $fs_3$}
\psfrag{fs31}{\tiny $fs_3s_1$} 
\psfrag{fs32}{\tiny $fs_3s_2$} 
\psfrag{fs323}{\tiny $fs_3s_2s_3$} 
\psfrag{fs3232}{\tiny $fs_3s_2s_3s_2$} 
\psfrag{fs321}{\tiny $fs_3s_2s_1$} 
\psfrag{fs3212}{\tiny $fs_3s_2s_1s_2$}  
\psfrag{fs2}{\tiny $fs_2$} 
\psfrag{fs21}{\tiny $fs_2s_1$} 
\psfrag{fs23}{\tiny $fs_2s_3$} 
\psfrag{M}{\scriptsize $M$} 
\psfrag{M-1}{\scriptsize $M-1$} 
\psfrag{M-2}{\scriptsize $M-2$} 
\psfrag{M-3}{\scriptsize $M-3$} 
\psfrag{a}{\small a)}
\psfrag{b}{\small b)}
\epsfig{file=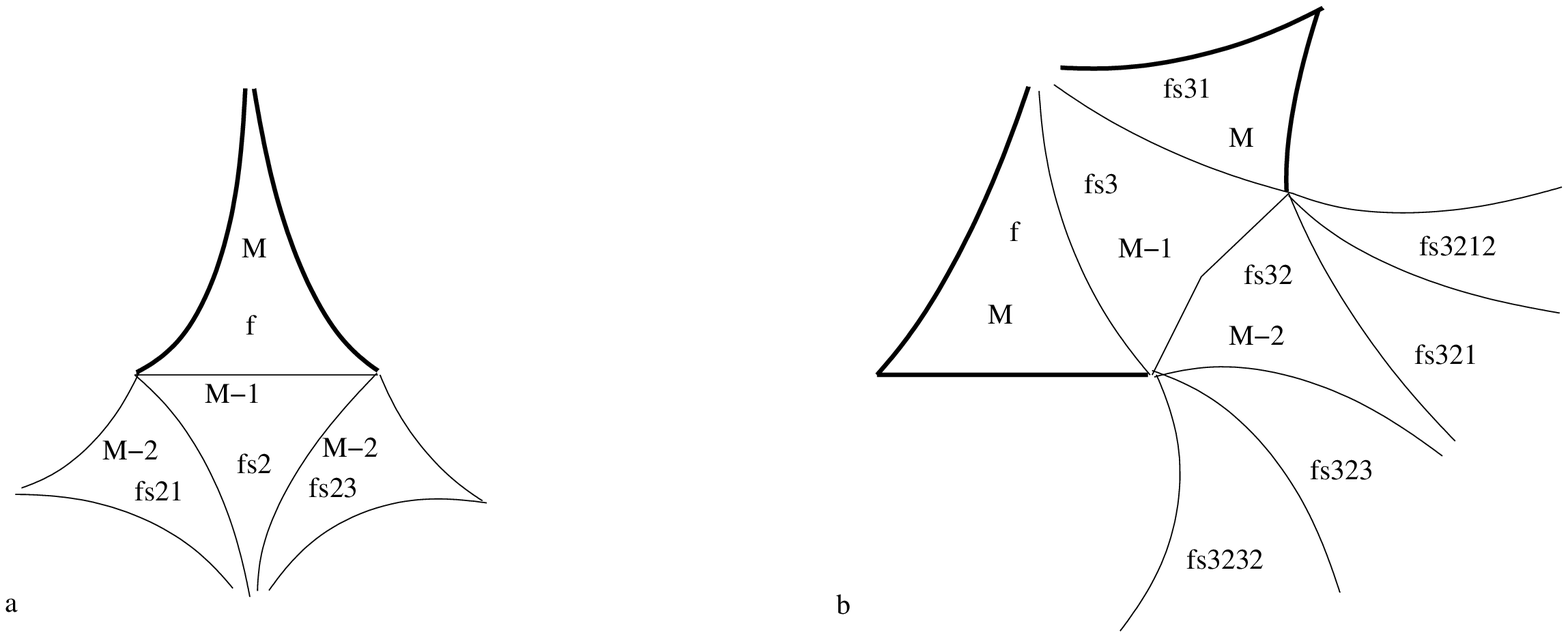,width=0.91\linewidth}
\caption{Towards the proof of Lemma~\ref{55inf}}  
\label{55inf2}
\end{center}
\end{figure}

So, one of $fs_1$ and $fs_3$ (say, $fs_3$) belongs to $I$. By Corollary~\ref{inf neighb}, this implies $l(fs_1)\ne M-1$, and hence $fs_1\notin I$. By Lemma~\ref{5angle}, $fs_3s_2\in I$, and $l(fs_3s_2)=M-2$. 
Taking in account the absence of interior vertices in $P$, this implies $l(fs_3s_1)=M$ (by Lemma~\ref{lemma7} applied to
$fs_3s_1,fs_3,fs_3s_2$, see Fig.\ref{55inf2}.b).
Furthermore, since $m_{12}m_{23}$ is not divisible by $2$ and $3$ (and since $fs_2\notin I$),
$fs_3s_2s_3\in I$, and $fs_3s_2s_3s_2\in I$. We will show that  $l(fs_3s_2s_3)=M-3$, and, similarly, $l(fs_3s_2s_1)=M-3$, which will contradict Corollary~\ref{inf neighb}.

Assume that $l(fs_3s_2s_3)=M-1$, then $l(fs_3s_2s_3s_2)=M$, otherwise by 
Lemma~\ref{lemma7} applied to $fs_3s_2$, $fs_3s_2s_3$ and $fs_3s_2s_3s_2$ we have $fs_2\in I$ which is false.
Thus, for the word $f'=fs_3s_2s_3s_2$ we have $l(f')=M$, $f',f's_2\in I$, which is already proven to be impossible.
Hence,  $l(fs_3s_2s_3)=M-3$, which completes the proof.
\end{proof}

\subsection{Rank $4$ groups}

In this section, we describe two series of rank $4$ groups without  finite index reflection subgroups. 
We keep all the notation from the previous section. 

\begin{lemma}
\label{555-4} 
Let $(W,S)$ be an odd-angled Coxeter system of rank $4$ with relations $(s_1s_i)^{5k_i}=(s_is_j)^{3k_{ij}}=1$, $2\le i,j\le 4$, where $k_i$ is not divisible by $3$ (see~Fig.~\ref{f555-42}). Then $W$ has no finite index reflection subgroups.

\end{lemma}

\begin{figure}[!h]
\begin{center}
\psfrag{5}{\scriptsize $5$}
\epsfig{file=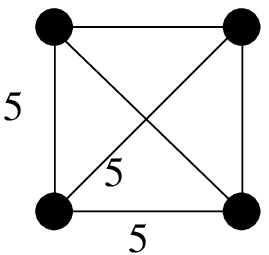,width=0.1205\linewidth}
\caption{$\Cox_{\div}(W)$, Lemma~\ref{555-4}}  
\label{f555-42}
\end{center}
\end{figure}

\begin{proof}
As usual, suppose $W$ has a  finite index reflection subgroup.

Suppose that $fs_1\in I$, i.e. $l(fs_1)=M-1$. Then, by Lemma~\ref{5angle}, $l(fs_1s_i)=M-2$ for $i=2,3,4$
(since $k_i$ has no prime divisor smaller than $5$). This contradicts Lemma~\ref{end-of-reduced-expression} and Corollary~\ref{ending-element-corollary}.

Therefore, $l(fs_1)=M+1$, and $fs_1\notin I$. Without loss of generality we assume that $fs_2\in I$, and $l(fs_2)=M-1$. 
By Lemma~\ref{5angle}, this implies $fs_2s_1\in I$,
$l(fs_2s_1)=M-2$. Since $k_i$  is not divisible by $3$, $fs_2s_1s_2,fs_2s_1s_2s_1\in I$ (see Fig.~\ref{f555-4}).
If $l(fs_2s_1s_2)=M-1$, then $l(fs_2s_1s_2s_1)=M$, which is impossible by the previous paragraph applied to the element 
$f'=fs_2s_1s_2s_1$ in place of $f$. Hence, $l(fs_2s_1s_2)=M-3$. 
This implies $l(fs_2s_3)=l(fs_2s_4)=M$, otherwise the elements
 $fs_2s_i, fs_2,fs_2s_1, fs_2s_1s_2$, where $i=3,4$, of lengths $M-2, M-1,M-2, M-3$ are in contradiction with 
Lemma~\ref{lemma b}. Reasoning as above (while showing $l(fs_2s_1s_2)=M-3$), we obtain $l(fs_2s_1s_3)=l(fs_2s_1s_4)=M-3$,
which (together with  $l(fs_2s_1)=M-2$ and $l(fs_2s_1s_2)=M-3$) is impossible by Lemma~\ref{end-of-reduced-expression} and Corollary~\ref{ending-element-corollary}.
\end{proof}

\begin{figure}[!h]
\begin{center}
\psfrag{f}{\tiny $f$} 
\psfrag{fs2}{\tiny $fs_2$}
\psfrag{fs21}{\tiny $fs_2s_1$} 
\psfrag{fs23}{\tiny $fs_2s_3$} 
\psfrag{fs212}{\tiny $fs_2s_1s_2$} 
\psfrag{fs213}{\tiny $fs_2s_1s_3$} 
\psfrag{fs321}{\tiny $fs_3s_2s_1$} 
\psfrag{fs2121}{\tiny $fs_2s_1s_2s_1$}  
\psfrag{fs2}{\tiny $fs_2$} 
\psfrag{fs21}{\tiny $fs_2s_1$} 
\psfrag{fs23}{\tiny $fs_2s_3$} 
\psfrag{M}{\scriptsize $M$} 
\psfrag{M-1}{\scriptsize $M-1$} 
\psfrag{M-2}{\scriptsize $M-2$} 
\psfrag{M-3}{\scriptsize $M-3$} 
\psfrag{a}{\small a)}
\psfrag{b}{\small b)}
\epsfig{file=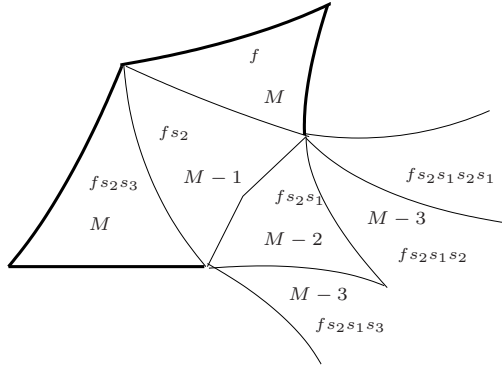,width=0.41\linewidth}
\caption{Towards the proof of Lemma~\ref{555-4}}  
\label{f555-4}
\end{center}
\end{figure}

\begin{lemma}
\label{remaining4}
Let $W=\langle s_1,\dots,s_4 \ | \ (s_is_j)^{m_{ij}}=1\rangle$ be an odd-angled Coxeter system of rank $4$ such that 
\begin{itemize}
\item $m_{12},m_{23},m_{34}\ne \infty$; 
\item $m_{12}$ and $m_{34}$ are not divisible by $3$.

\end{itemize}
Then $W$ has no finite index reflection subgroups.

\end{lemma}

To prove Lemma~\ref{remaining4} suppose that $W$ has a  finite index reflection subgroup $V$.
 
We say that the elements $ws_is_j,ws_i,w,ws_j,ws_js_i$ compose an {\it $(i,j)$-symmetric $5$-tuple} (or simply {\it symmetric $5$-tuple}) if all these elements lie in $I$ and are of length $M,M-1,M-2,M-1,M$ respectively.

The plan of the proof is the following. First, we show (Claims~\ref{non or sym}--\ref{exists}) that there is at least one symmetric $5$-tuple. Then we prove that the existence of one symmetric $5$-tuple implies the existence of an infinite number of them. Since, by definition, all symmetric $5$-tuples lie in $I$, we obtain a contradiction with the finiteness of the index of $V$.

\begin{claim}
\label{non or sym}
Suppose that $m_{ij}, m_{kn}\ne \infty$ and $m_{ij}m_{kn}$ is not divisible by $3$, { where $\{i,j,k,n\}=\{1,2,3,4\}$}.
Let $w,ws_j\in I$, $l(w)=M-2=l(ws_j)-1$. Then either $ws_js_k\notin I$ and $ws_js_n\notin I$, or  
 $ws_js_k\in I$, $ws_js_n\in I$ and $ws_js_k,ws_j,ws_js_n$ are contained in some symmetric $5$-tuple.

\end{claim}

\begin{proof}
We consider two cases: either $ws_js_k\notin I$ and $ws_js_n\in I$, or both $ws_js_k$ and $ws_js_n$ belong to $I$. The case  $ws_js_n\notin I$ and $ws_js_k\in I$ is identical to the former.

Suppose that $ws_js_k\notin I$, $ws_js_n\in I$. Since $ws_j\in I$ and $m_{kn}$ is not divisible by $3$, we have
 $ws_j,ws_js_n,ws_js_ns_k,ws_js_ns_ks_n,ws_js_ns_ks_ns_k\in I$ 
(see Fig~\ref{fig non or sym}.a). If $l(ws_js_n)=M$,
then $l(ws_js_ns_k)=M-1$ and Lemma~\ref{lemma7} (applied to the elements $ws_j,ws_js_n,ws_js_ns_k$ of the lengths 
$M-1,M,M-1$) implies that $ws_js_k\in I$, which is false by the assumption. So, $l(ws_js_n)=M-2$.
Similarly, $l(ws_js_ns_k)=M-3$. However, in this case the elements $w,ws_j,ws_js_n,ws_js_ns_k$ (of lengths $M-2,M-1,M-2,M-3$
respectively) are in contradiction with Lemma~\ref{lemma b}.
Therefore, either $ws_js_k\notin I$ and $ws_js_n\notin I$, or $ws_js_k\in I$ and $ws_js_n\in I$.

Now suppose that $ws_js_k\in I$, $ws_js_n\in I$.
As all subgroups generated by three or more different elements of $S$ are infinite, 
Corollary \ref{ending-element-corollary} and Lemma \ref{end-of-reduced-expression} imply that at most 2 neighbors 
of $ws_j$ can have length $l(ws_j)-1=M-2$.
Moreover, both $ws_js_k$ and $ws_js_n$ cannot be of length $M$ simultaneously: indeed, in this case Lemma~\ref{lemma7} implies that neither $ws_js_ns_k$ nor $ws_js_ks_n$ 
belong to $I$, which is impossible since $m_{kn}$ is not divisible by $3$. 

Therefore, without loss of generality we may assume that $l(ws_js_k)=M$, $l(ws_j)=M-1$, $l(ws_js_n)=M-2$ 
(see Fig~\ref{fig non or sym}.b).
By Lemma~\ref{lemma b} (applied to the elements $w,ws_j,ws_js_n,ws_js_ns_k$) we see that $l(ws_js_ns_k)\ne M-3$.
So, $l(ws_js_ns_k)=M-1$. Furthermore, Lemma~\ref{lemma7} (applied to the elements $ws_js_n,ws_js_ns_k,ws_js_ns_ks_n$) implies that $l(ws_js_ns_ks_n)\ne M-2$, i.e. $l(ws_js_ns_ks_n)= M$, and the elements $ws_js_k,ws_j,ws_js_n,ws_js_ns_k,ws_js_ns_ks_n$ compose a symmetric $5$-tuple. 
\end{proof}

\begin{figure}[!h]
\begin{center}
\psfrag{f}{\tiny $w$} 
\psfrag{fsj}{\tiny $ws_j$} 
\psfrag{fsjk}{\tiny $ws_js_k$} 
\psfrag{fsjn}{\tiny $ws_js_n$}
\psfrag{fsjnk}{\tiny $ws_js_ns_k$}
\psfrag{fsjnkn}{\tiny $ws_js_ns_ks_n$}
\psfrag{M}{\tiny $M$} 
\psfrag{M-1}{\tiny $M-1$} 
\psfrag{M-2}{\tiny $M-2$} 
\psfrag{M-3}{\tiny $M-3$} 
\psfrag{a}{\small a)}
\psfrag{b}{\small b)}
\epsfig{file=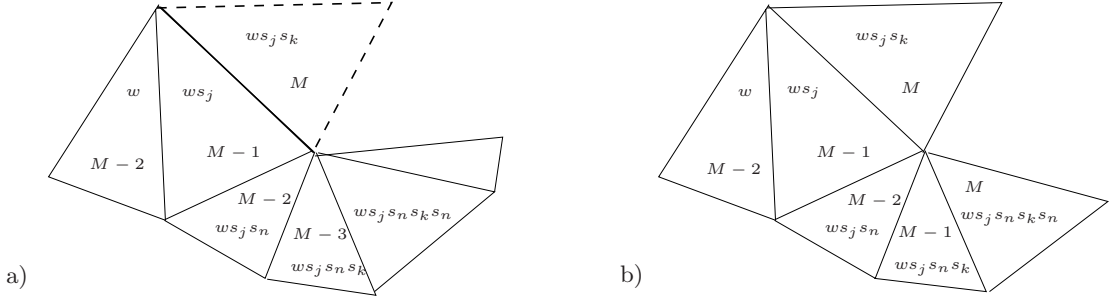,width=0.91\linewidth}
\caption{Towards the proof of Lemma~\ref{non or sym}}  
\label{fig non or sym}
\end{center}
\end{figure}

\begin{claim}
\label{symm}
Suppose that $w,ws_j\in I$, $l(w)=M-2$, $l(ws_j)=l(ws_i)=M-1$. If $m_{ij}\ne\infty$ is not divisible by $3$, then 
$ws_i\in I$, and $ws_is_j,ws_i,w,ws_j,ws_js_i\in I$ compose a symmetric $5$-tuple in $I$.

\end{claim}

\begin{proof}
First, suppose $ws_i\notin I$. Then, since $m_{ij}$ is not divisible by $3$, $ws_js_i,ws_js_is_j\in I$ (see Fig.~\ref{symm_reproduce}.a). Furthermore, due to Lemma~\ref{lemma7} we have $l(ws_js_i)=M$, and by the same reason $l(ws_js_is_j)=M+1$, which is impossible since $ws_js_is_j\in I$.

Thus, $ws_i\in I$. Applying Lemma~\ref{lemma7} again, we see that $l(ws_is_j)=M=l(ws_js_i)$ and $l(ws_is_js_i)=M+1=l(ws_js_is_j)$, which implies $ws_is_j,ws_js_i\in I$ and  
 $ws_is_j,ws_i,w,ws_j,ws_js_i\in I$ form a symmetric $5$-tuple in $I$.
\end{proof}

\begin{figure}[!h]
\begin{center}
\psfrag{f}{\tiny $w$} 
\psfrag{fsi}{\tiny $ws_i$} 
\psfrag{fsj}{\tiny $ws_j$} 
\psfrag{fsij}{\tiny $ws_is_j$} 
\psfrag{fsjk}{\tiny $ws_js_k$} 
\psfrag{fsji}{\tiny $ws_js_i$}
\psfrag{fsjij}{\tiny $ws_js_is_j$}
\psfrag{fsik}{\tiny $ws_is_k$}
\psfrag{fsjk}{\tiny $ws_js_k$}
\psfrag{M}{\tiny $M$} 
\psfrag{M-1}{\tiny $M-1$} 
\psfrag{M+1}{\tiny $M+1$} 
\psfrag{M-2}{\tiny $M-2$} 
\psfrag{M-3}{\tiny $M-3$} 
\psfrag{a}{\small a)}
\psfrag{b}{\small b)}
\epsfig{file=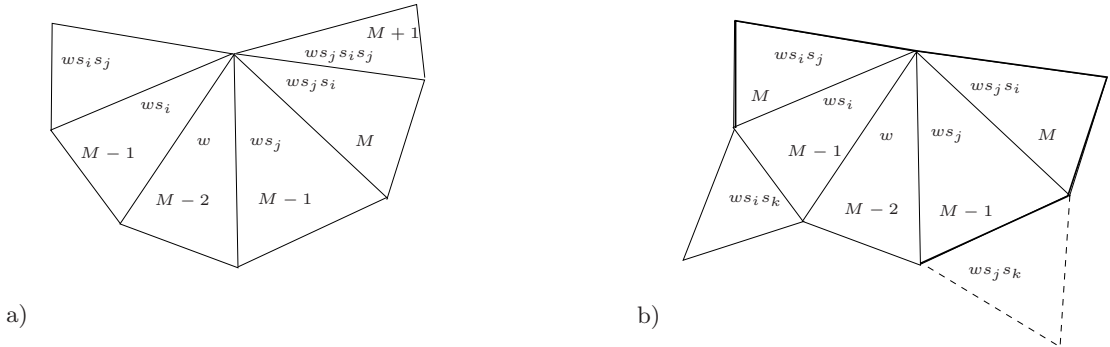,width=0.91\linewidth}
\caption{Towards the proof of (a) Claim~\ref{symm} and (b) Claim~\ref{reproduce}  }  
\label{symm_reproduce}
\end{center}
\end{figure}

\begin{claim}
\label{reproduce}
Let either $(i,j,k,n)=(1,2,3,4)$ or $(i,j,k,n)=(4,3,2,1)$. 
Suppose that an element $w$ of length $M-2$ is included in an $(i,j)$-symmetric $5$-tuple. 
Then each of the elements $ws_i$ and $ws_j$  is contained in some $(k,n)$-symmetric $5$-tuple.

In particular,
all four neighbors of $ws_i$ (and $ws_j$) are contained in $I$, 
and there are precisely two neighbors of $ws_i$ (and $ws_j$) of length $M$ 
and two neighbors of length $M-2$. 

\end{claim}

\begin{proof}
By Lemma~\ref{lemma a}, each of the elements $ws_js_i$ and $ws_is_j$ has exactly one neighbor in $I$.
Since $ws_i,ws_is_j\in I$ and $m_{jk}=m_{23}\ne \infty$, this implies $ws_is_k\in I$ (see Fig.~\ref{symm_reproduce}.b). 
Applying Claim~\ref{non or sym} we see that
$ws_i$ lies in a $(k,n)$-symmetric $5$-tuple, and all four neighbors of $ws_i$ are contained in $I$ and have the desired length. If all neighbors of $ws_j$ are in $I$, then by Claim~\ref{non or sym} 
$ws_j$ lies in a $(k,n)$-symmetric $5$-tuple, and the claim is proved. 

Suppose that a neighbor of $ws_j$ does not belong to $I$. Then, again by Claim~\ref{non or sym}, $ws_js_k,ws_js_n\notin I$. Recall also that 
$ws_js_is_k,ws_js_is_n\notin I$, but $ws_j,ws_js_i\in I$.
In view of Lemma~\ref{inf}, this implies that $m_{ik}=m_{in}=\infty$. On the other hand, we have already proved that $ws_i$ is contained in a $(k,n)$-symmetric $5$-tuple, so either $ws_is_k$ or $ws_is_n$ has length $M-2$. Therefore, either elements $w,ws_i,ws_is_k$ or elements $w,ws_i,ws_is_n$ are of lengths $M-2,M-1,M-2$ respectively. In view of Corollary~\ref{inf neighb}, this contradicts  $m_{ik}=m_{in}=\infty$. 
\end{proof}

\begin{claim}
\label{exists}
There exists a $5$-tuple of elements in $I$ composing an $(i,j)$-symmetric $5$-tuple with $\{i,j\}=\{1,2\}$ or $\{3,4\}$.

\end{claim}

\begin{proof}
Let $fs_i\in I$ be a neighbor of $f\in I$, $l(fs_i)=M-1=l(f)-1$.
Choose $j$ so that $\{i,j\}=\{1,2\}$ or  $\{i,j\}=\{3,4\}$ { (so, we need to construct either $(i,j)$- or $(k,n)$-symmetric $5$-tuple)}, in particular   $m_{ij}$ is finite and is not divisible by 
$3$. This implies that $I$ contains at least $5$ elements from the set 
$I_0=\{f(s_is_j)^m, fs_j(s_is_j)^m, \ m\in \Z\}$.
We consider two cases: either $I_0\subset I$ or not.

First, suppose  $I_0\subset I$. Then $l(fs_i)=l(fs_j)=M-1$. Hence, $l(fs_k)=l(fs_n)=M+1$, and $fs_k, fs_n \notin I$.    If either $fs_i$ or $fs_j$ has a neighbor in $I$ other than
$f$, $fs_is_j$ and $fs_js_i$, then the statement follows from Claim~\ref{non or sym}. 
So, we assume that $fs_i$ and $fs_j$ has no other neighbors in $I$. Then Lemma~\ref{inf} implies $m_{ik}=m_{in}=m_{jk}=m_{jn}=\infty$,
which contradicts the assumption that $m_{23}\ne \infty$.

Now, suppose  $I_0\setminus I\ne \emptyset$. If $I_0$ contains a symmetric $5$-tuple, then there
 is nothing to prove. So, we suppose that $I_0$ contains no symmetric $5$-tuple. Since $I_0\setminus I\ne \emptyset$, 
Lemma~\ref{lemma7} implies that $l(fs_j)=M+1$, $fs_j\notin I$. Thus, $f,fs_i,fs_is_j,fs_is_js_i,fs_is_js_is_j\in I$ 
(here we use the assumption that $m_{ij}$ is not divisible by $3$). Recall that $M=l(f)=l(fs_i)+1=l(fs_is_j)+2$ 
(the last equality follows from Lemma~\ref{5angle}). 
Furthermore, $l(fs_is_js_i)=M-3$ (otherwise, $l(fs_is_js_i)=M-1$, and either
$l(fs_is_js_is_j)=M$ and we obtain an $(i,j)$-symmetric $5$-tuple in $I_0$, or $l(fs_is_js_is_j)=M-2$ and by Lemma~\ref{lemma7} we have
 $I_0\subset I$ which contradicts our assumptions). 

By Claim~\ref{non or sym}, either $fs_i$ belongs to a $(k,n)$-symmetric $5$-tuple,
or $fs_i$ has no neighbors in $I$ except $f$ and $fs_is_j$. In the former case there is nothing to prove, so suppose the 
latter. By Lemma~\ref{lemma a}, $f$ has no other neighbors in $I$ except $fs_i$. So, by Lemma~\ref{inf},  
$m_{ik}=m_{in}=\infty$. Similarly, if neither $fs_is_js_k$ nor $fs_is_js_n$ belongs to $I$, then $m_{jk}=m_{jn}=\infty$
which is false since $m_{23}\ne \infty$ (see Fig.~\ref{f exists}). 
Hence, at least one of  $fs_is_js_k$ and $fs_is_js_n$  belongs to $I$. Without loss
of generality we may assume $fs_is_js_k\in I$. Note, that since $l(fs_is_js_i)=M-3$, $l(fs_is_j)=M-2$ and $m_{ik}=\infty$, by Corollary~\ref{inf neighb} we have 
$l(fs_is_js_k)\ne M-3$. So,  $l(fs_is_js_k)=M-1$,
and, similarly, $l(fs_is_js_n)=M-1$. Thus, by Claim~\ref{symm} 
 $fs_is_js_ns_k,fs_is_js_n,fs_is_j,fs_is_js_k,fs_is_js_ks_n$ compose a $(k,n)$-symmetric $5$-tuple, and the lemma is proved.
\end{proof}

\begin{figure}[!h]
\begin{center}
\psfrag{f}{\tiny $f$} 
\psfrag{fsi}{\tiny $fs_i$} 
\psfrag{fsij}{\tiny $fs_is_j$} 
\psfrag{fsiji}{\tiny $fs_is_js_i$} 
\psfrag{fsijij}{\tiny $fs_is_js_is_j$} 
\psfrag{fsijk}{\tiny $fs_is_js_k$}
\psfrag{fsjnk}{\tiny $fs_js_ns_k$}
\psfrag{fsjnkn}{\tiny $fs_js_ns_ks_n$}
\psfrag{fsjnjn}{\tiny $fs_js_ns_js_n$}
\psfrag{M}{\tiny $M$} 
\psfrag{M-1}{\tiny $M-1$} 
\psfrag{M-2}{\tiny $M-2$} 
\psfrag{M-3}{\tiny $M-3$} 
\epsfig{file=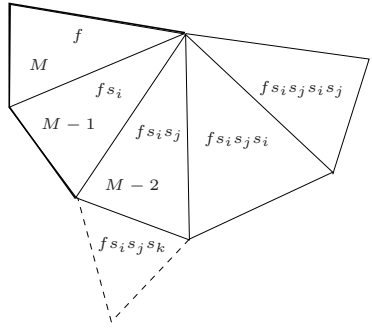,width=0.3\linewidth}
\caption{Towards the proof of Claim~\ref{exists}}  
\label{f exists}
\end{center}
\end{figure}

Now, we are able to prove Lemma~\ref{remaining4}.

\begin{proof}[Proof of Lemma~\ref{remaining4}]
By Claim~\ref{exists}, $I$ contains an $(i,j)$-symmetric $5$-tuple, say $ws_is_j,ws_i,w,ws_j,ws_js_i$, where $\{i,j\}=\{1,2\}$ or  $\{i,j\}=\{3,4\}$. 
By Claim~\ref{reproduce}, $ws_i$ belongs to a $(k,n)$-symmetric $5$-tuple lying in $I$ (where $\{k,n\}=\{1,2,3,4\}\setminus \{i,j\})$. Thus, either $ws_is_ks_n$ or $ws_is_ns_k$ has length $M-1$ and belongs again to 
$(i,j)$-symmetric $5$-tuple lying in $I$.

By successive repetitions of this argument, we obtain that for every $p \in \N$ there exists 
$v_1, \hdots, v_p \in \{s_is_j, s_js_i\}$ and $u_1, \hdots, u_p \in \{s_ks_n, s_ns_k\}$ such that 
$wv_1u_1v_2u_2 \cdots v_pu_p \in I$. By Theorem \ref{M-operation}, the expression $v_1u_1v_2u_2 \cdots v_pu_p$ is reduced, and therefore $ws_iv_1u_1v_2u_2 \cdots v_pu_p \neq ws_iv_1u_1v_2u_2 \cdots v_{p'}u_{p'}$ for $p \neq p'$. 
Hence, we can construct arbitrary many distinct elements contained in $I$, which contradicts the assumption on finiteness of $I$. Thus, $W$ has no finite index reflection subgroup.
\end{proof}

\subsection{Higher rank groups}

In this section, we point out a series of groups of rank greater than four without  finite index reflection subgroups.

\begin{lemma}
\label{k3-3l}
Let $W$ be the group with $\Cox_{\div}(W)$ shown in Fig.~\ref{f_k3-3l}.
Then $W$ has no finite index reflection subgroups.

\end{lemma}

\begin{figure}[!h]
\begin{center}
\psfrag{k}{\scriptsize $k$}
\psfrag{l}{\scriptsize $l$}
\psfrag{p}{\scriptsize $3,\infty$}
\epsfig{file=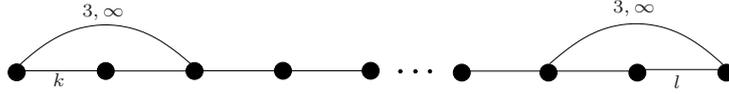,width=0.6\linewidth}
\caption{$\Cox_{\div}(W)$ for a series of groups without  finite index reflection subgroups, \ $k,l\ge 5$}
\label{f_k3-3l}
\end{center}
\end{figure}

\begin{proof}
As usual, we assume that there is a  finite index reflection subgroup $V\subset W$. We keep all the notation from the previous sections.

We label the vertices of  $\Cox_{\div}(W)$ from left to right, so that the least divisor of $m_{12}$ is $k\ge 5$, the least divisor of $m_{n-1,n}$ is $l\ge 5$, and $m_{i,i+1}\in 3\Z_{\mathrm{odd}}$ for $i=2,\dots n-2$, $m_{13},m_{n-2,n}\in \{3\Z_{\mathrm{odd}},\infty \}$, $m_{ij}=\infty$ for all other pairs $i,j$.

To prove the lemma, we take an element $f$ of maximal length $M$ in $I$ and show that all its neighbors have length $M+1$, i.e. there is no geodesic from $f$ to $1$. We consider the neighbors $fs_i$ for $i\in [3,n-2]$ in Claim~\ref{M3n-2}, for $i=2,n-2$ in Claim~\ref{M2n-1}, and for $i=1,n$ in Claim~\ref{M1n}. { Note that the assertion $l(fs_i)\ne M-1$ is equivalent to $fs_i\notin I$ due to convexity of $I$, see Lemma~\ref{geodesic-gallery-lemma-2}.}

\begin{claim}
\label{M3n-2}
 $l(fs_i)\ne M-1$ for $3\le i\le n-2$.
\end{claim}

\begin{proof}
Suppose the contrary, i.e. there exists $i$, $3\le i\le n-2$, such that $l(fs_i)=M-1$.
Consider the elements $fs_is_j$.
Suppose that $l(fs_is_j)=M-2$ for some $j>i$.
Then, since $m_{tj}=\infty$ for all $t<i$, Corollary~\ref{inf neighb} implies that 
\begin{equation}
\label{M}
l(fs_is_t)=M \quad {\text{ for all}}  \quad t<i, 
\end{equation}
see Fig.~\ref{cl1}.
Furthermore, by Lemma~\ref{5angle} (applied to the elements $f,fs_i$ and $fs_is_{i-1}$), we see that $fs_is_{i-1}\in I$.
Similarly, applying   Lemma~\ref{5angle} to the elements $fs_is_{t+1},fs_i$ and $fs_is_t$ we see that $fs_is_t\in I$
for all $t<i$. In particular,  $fs_is_2,fs_is_1\in I$. Since $m_{12}$ is not divisible by $3$, Lemma~\ref{5angle} 
implies also that $l(fs_is_1)=M-2$ which contradicts~(\ref{M}).

The contradiction implies that $l(fs_is_j)=M$ for all $j>i$. Similarly, we can prove that $l(fs_is_j)=M$ for all $j<i$. Thus, all the neighbors of $fs_i$ have length $M$ which is clearly impossible, so the claim follows.
\end{proof}

\begin{figure}[!h]
\begin{center}
\psfrag{f}{\scriptsize $f$} 
\psfrag{fsi}{\scriptsize $fs_i$} 
\psfrag{fsij}{\scriptsize $fs_is_j$} 
\psfrag{fsiji}{\tiny $fs_is_js_i$} 
\psfrag{fsijij}{\tiny $fs_is_js_is_j$} 
\psfrag{fsijk}{\tiny $fs_is_js_k$}
\psfrag{fsjnk}{\tiny $fs_js_ns_k$}
\psfrag{fsjnkn}{\tiny $fs_js_ns_ks_n$}
\psfrag{fsjnjn}{\tiny $fs_js_ns_js_n$}
\psfrag{M}{\small $M$} 
\psfrag{M1}{\small $M-1$} 
\psfrag{M2}{\small $M-2$} 
\psfrag{M-3}{\tiny $M-3$} 
\epsfig{file=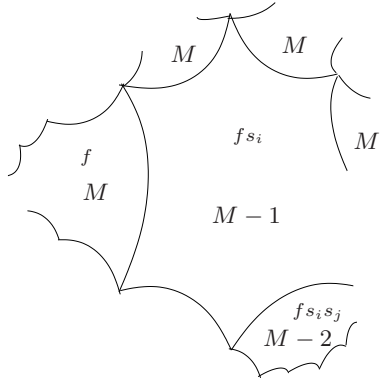,width=0.3\linewidth}
\caption{Towards the proof of Claim~\ref{M3n-2}}  
\label{cl1}
\end{center}
\end{figure}

\begin{claim}
\label{M2n-1}
  $l(fs_2)\ne M-1$, $l(fs_{n-1})\ne M-1$.
\end{claim}

\begin{proof} 
We will prove that $l(fs_2)\ne M-1$, the second part of the claim follows by symmetry.
Suppose that  $l(fs_2)= M-1$, i.e. $fs_2\in I$. 

By Lemma~\ref{5angle}, $fs_2s_1\in I$ and $l(fs_2s_1)=M-2$ (since $m_{12}$ is not divisible by $3$). By Lemma~\ref{inf neighb}, this implies $l(fs_2s_j)=M$ for all $j> 3$ (in particular, $fs_2s_j\notin I$ for $3<j\le n-2$ by Claim~\ref{M3n-2}). The element $fs_2s_3$ may be either of length $M$ or of length $M-2$.

\begin{figure}[!h]
\begin{center}
\psfrag{f}{\scriptsize $f$} 
\psfrag{fs2}{\scriptsize $fs_2$} 
\psfrag{fs1}{\scriptsize $fs_2s_1$} 
\psfrag{fs23}{\scriptsize $fs_2s_3$} 
\psfrag{fs24}{\tiny $fs_2s_4$} 
\psfrag{fs25}{\tiny $fs_2s_5$} 
\psfrag{M}{\small $M$} 
\psfrag{M1}{\small $M-1$} 
\psfrag{M2}{\small $M-2$} 
\psfrag{M-3}{\tiny $M-3$} 
\epsfig{file=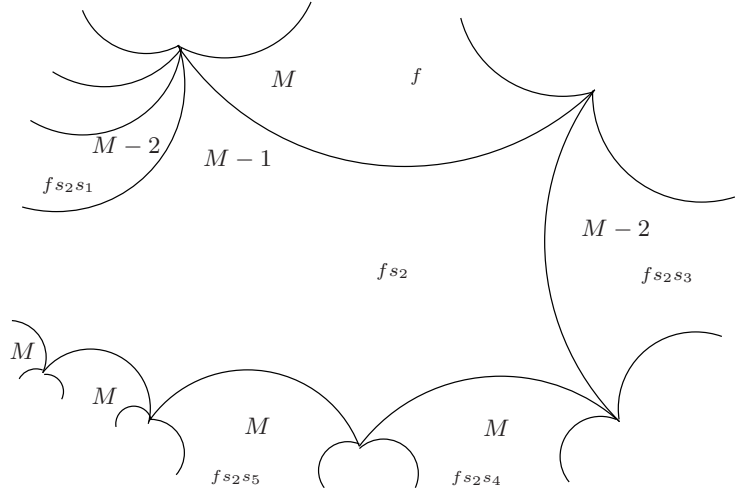,width=0.6\linewidth}
\caption{Towards the proof of Claim~\ref{M2n-1}}  
\label{cl5-22-1}
\end{center}
\end{figure}

According to Claim~\ref{M3n-2}, $fs_3\notin I$, which implies $fs_2s_3\in I$. By Claim~\ref{M3n-2}, $l((fs_2)s_3)\ne M$, so $l(fs_2s_3)= M-2$, see Fig.~\ref{cl5-22-1}. In view of Corollary~\ref{inf neighb}, this implies $m_{13}\in 3\Z_{\mathrm{odd}}$.
Applying Lemma~\ref{lemma7} to the elements $fs_2s_1$, $fs_2$ and $fs_2s_3$ of lengths $M-2, M-1$ and $M-2$, we obtain
$fs_2(s_3s_1)^k, fs_2s_3(s_1s_3)^k\in I$ $\forall k\in\Z$, and the length of all these $2m_{13}$ elements can be computed according to  Lemma~\ref{lemma7}. In particular, $l(fs_2s_3s_1)=M-3$. 

Let us figure out the lengths of the elements $fs_2s_3s_i$ for $i>2$, see Fig.~\ref{cl5-22-2}. For this, observe that $m_{1i}=\infty$ for $i>3$, and
thus we have $l(fs_2s_3s_i)=M-1$, otherwise we get a contradiction with Corollary~\ref{inf neighb}. For $i=3$ we also get $M-1$ by our assumption. 

\begin{figure}[!hb]
\begin{center}
\psfrag{f}{\scriptsize $f$} 
\psfrag{fs2}{\scriptsize $fs_2$} 
\psfrag{fs23}{\scriptsize $fs_2s_3$} 
\psfrag{fs24}{\scriptsize $fs_2s_4$} 
\psfrag{fs25}{\scriptsize $fs_2s_5$} 
\psfrag{fs26}{\scriptsize $fs_2s_6$}
\psfrag{fs234}{\scriptsize $fs_2s_3s_4$} 
\psfrag{fs235}{\scriptsize $fs_2s_3s_5$} 
\psfrag{fs236}{\scriptsize $fs_2s_3s_6$} 
\psfrag{fs21}{\scriptsize $fs_2s_1$} 
\psfrag{fs231}{\scriptsize $fs_2s_3s_1$} 
\psfrag{fs23n-2}{\scriptsize $fs_2s_3s_{n-2}$}
\psfrag{M}{\small $M$} 
\psfrag{M1}{\small $M-1$} 
\psfrag{M2}{\small $M-2$} 
\psfrag{M3}{\small $M-3$} 
\psfrag{M-3}{\tiny $M-3$} 
\epsfig{file=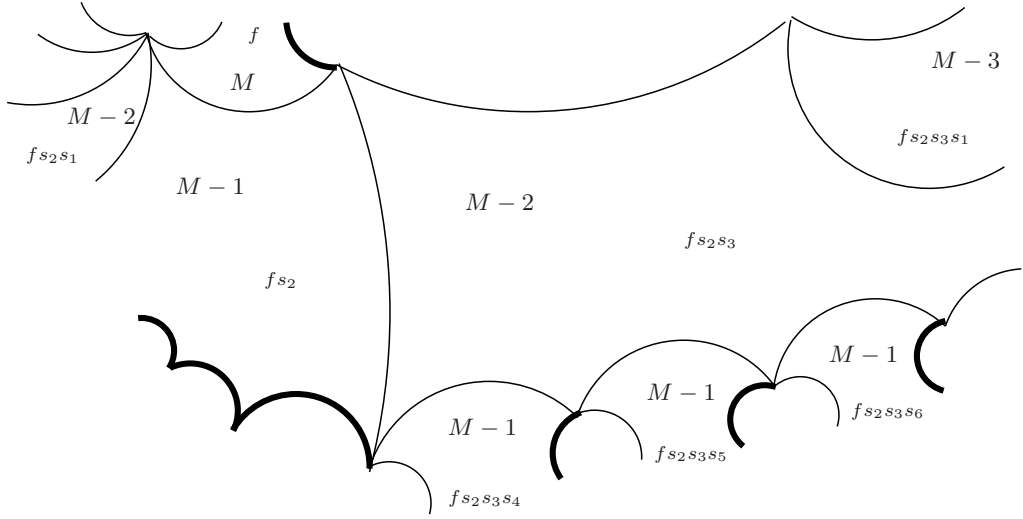,width=0.9\linewidth}
\caption{Towards the proof of Claim~\ref{M2n-1}}  
\label{cl5-22-2}
\end{center}
\end{figure}

Now we want to show that the elements $fs_2s_3s_i$ belong to $I$ for $2<i<n-1$. We will prove this by induction on $i$. For $i=3$ the statement is obvious. Assume the statement is true for $i-1$. Since $m_{i-1,i}$ is a multiple of $3$, either $fs_2s_3s_{i-1}s_i$ or $fs_2s_3s_i$ is in $I$. In the former case, $l(fs_2s_3s_{i-1}s_i)=M-2$ by Claim~\ref{M3n-2}, and we have $fs_2s_3s_i \in I$ by Lemma~\ref{lemma7}. So, in both cases $fs_2s_3s_i \in I$.

In particular, we get  $fs_2s_3s_{n-2} \in I$. Let us show that $fs_2s_3s_{n-2}s_{n-1} \notin I$. Suppose the contrary, i.e. $fs_2s_3s_{n-2}s_{n-1}\in I$. Since $l(fs_2s_3s_{n-2})=M-1$, the length of $fs_2s_3s_{n-2}s_{n-1}$ is either $M$ or $M-2$. If $l(fs_2s_3s_{n-2}s_{n-1})=M-2$, then the four elements $fs_2s_3s_{n-2}s_{n-1}$, $fs_2s_3s_{n-2}$, $fs_2s_3$, $fs_2s_3s_1$ of lengths $M-2$, $M-1$, $M-2$, $M-3$ are in contradiction with Lemma~\ref{lemma b}. Therefore, $l(fs_2s_3s_{n-2}s_{n-1})=M$, and  Lemma~\ref{5angle} implies that $fs_2s_3s_{n-2}s_n$ also belongs to $I$, and its length is $M-2$. Thus, we obtain four elements
$fs_2s_3s_{n-2}s_{n}, fs_2s_3s_{n-2}, fs_2s_3, fs_2s_3s_{n}$ of lengths $M-2,M-1, M-2,M-1$ respectively, and  $fs_2s_3s_{n-2} \in I$, see Fig.~\ref{cl5-22-3}. This contradicts Lemma~\ref{lemma7}.

\begin{figure}[!h]
\begin{center}
\psfrag{Fs231}{\scriptsize $fs_2s_3s_1$} 
\psfrag{fs23n}{\scriptsize $fs_2s_3s_n$} 
\psfrag{fs23}{\scriptsize $fs_2s_3$} 
\psfrag{fs23nn-2}{\tiny $fs_2s_3s_ns_{n-2}$} 
\psfrag{fs23nn-1}{\tiny $fs_2s_3s_ns_{n-1}$} 
\psfrag{fs23n-1n-2}{\tiny $fs_2s_3s_{n-1}s_{n-2}$} 
\psfrag{fs23n-1}{\tiny $fs_2s_3s_{n-1}$}
\psfrag{fs23n-2}{\scriptsize $fs_2s_3s_{n-2}$}
\psfrag{f23n-2n}{\tiny $fs_2s_3s_{n-2}s_n$}
\psfrag{fs23n-2n-1}{\scriptsize $fs_2s_3s_{n-2}s_{n-1}$}
\psfrag{M}{\small $M$} 
\psfrag{M1}{\small $M-1$} 
\psfrag{M2}{\small $M-2$}
\psfrag{M3}{\small $M-3$}  
\psfrag{M-3}{\tiny $M-3$} 
\epsfig{file=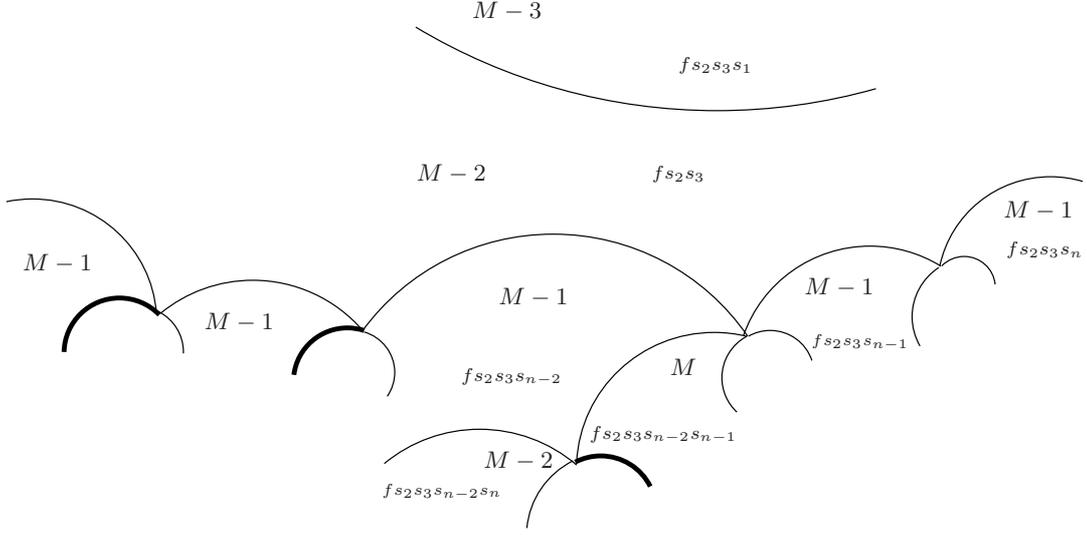,width=0.9\linewidth}
\caption{Towards the proof of Claim~\ref{M2n-1}}  
\label{cl5-22-3}
\end{center}
\end{figure}

Therefore,  $fs_2s_3s_{n-2}s_{n-1} \notin I$. This implies  $fs_2s_3s_{n-1} \in I$  (see Fig.~\ref{cl5-22-3}). Thus, we have two ``neighboring elements'' $fs_2s_3s_{n-1}$ and $fs_2s_3$ in $I$. Considering the elements obtained from $fs_2s_3$ by multiplication by $s_{n-1}$ and $s_n$, we see that at least five such elements belong to $I$ (since $m_{n-1,n}$ is not divisible by $3$). As it was proved above, $l(fs_2s_3s_{n-1})=l(fs_2s_3)+1=l(fs_2s_3s_{n})=M-1$, which means that we have an $(n,n-1)$-symmetric $5$-tuple $(fs_2s_3s_{n-1}s_n, fs_2s_3s_{n-1}, fs_2s_3, fs_2s_3s_n, fs_2s_3s_ns_{n-1})$.

Consider now the elements obtained from $fs_2s_3$ by multiplication by $s_{n-2}$ and $s_n$ (see Fig.~\ref{cl5-22-4}). Since the elements $fs_2s_3s_{n-2}, fs_2s_3, fs_2s_3s_{n}$ all belong to $I$ and have lengths $M-1,M-2,M-1$, we see that $l(fs_2s_3s_{n}s_{n-2})=M$, which implies $fs_2s_3s_{n}s_{n-2}\notin I$ by Claim~\ref{M3n-2}. In view of $fs_2s_3s_{n},fs_2s_3s_ns_{n-1}\in I$, this implies $fs_2s_3s_{n}s_{n-1}s_{n-2}\in I$, and thus $l(fs_2s_3s_{n}s_{n-1}s_{n-2})=M-1$. Since  $fs_2s_3s_{n}s_{n-1}\in I$ and $l(fs_2s_3s_{n}s_{n-1})=M$, this contradicts Claim~\ref{M3n-2}, which completes the proof of the claim. 
\end{proof}

\begin{figure}[!h]
\begin{center}
\psfrag{fs23n}{\scriptsize $fs_2s_3s_n$} 
\psfrag{fs23}{\scriptsize $fs_2s_3$} 
\psfrag{fs23nn-2}{\tiny $fs_2s_3s_ns_{n-2}$} 
\psfrag{fs23nn-1}{\tiny $fs_2s_3s_ns_{n-1}$} 
\psfrag{fs23n-1n-2}{\tiny $fs_2s_3s_{n-1}s_{n-2}$} 
\psfrag{fs23n-2n-1}{\tiny $fs_2s_3s_{n-1}s_{ n}$} 
\psfrag{fs23n-1}{\scriptsize $fs_2s_3s_{n-1}$}
\psfrag{fs23n-2}{\scriptsize $fs_2s_3s_{n-2}$}
\psfrag{M}{\small $M$} 
\psfrag{M1}{\small $M-1$} 
\psfrag{M2}{\small $M-2$} 
\psfrag{M-3}{\tiny $M-3$} 
\epsfig{file=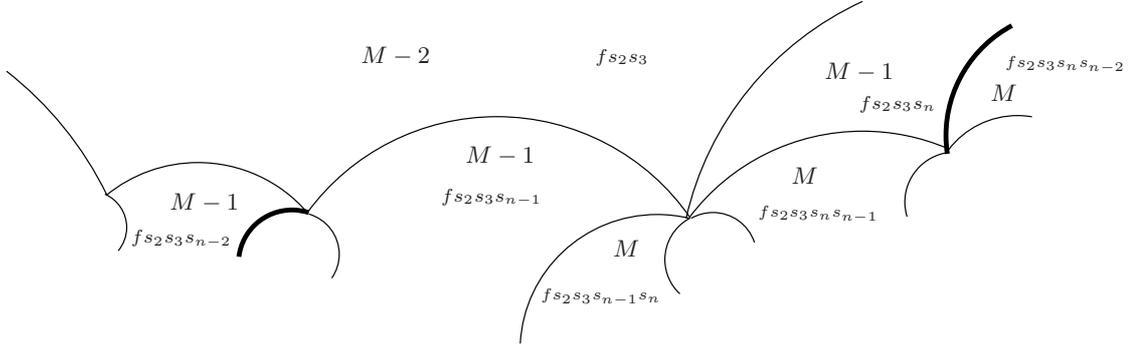,width=0.98\linewidth}
\caption{Towards the proof of Claim~\ref{M2n-1}}  
\label{cl5-22-4}
\end{center}
\end{figure}

\begin{claim}
\label{M1n} 
 $l(fs_1)\ne M-1$, $l(fs_{n})\ne M-1$.
\end{claim}

\begin{proof}
 We will show that  $l(fs_1)\ne M-1$, the second part is similar.
Suppose that  $l(fs_1)= M-1$.
By Lemma~\ref{5angle}, $fs_1s_2\in I$ and $l(fs_1s_2)=M-2$.
Furthermore, $fs_1s_2s_1,fs_1s_2s_1s_2\in I$ ($m_{12}$ is not divisible by $3$). If $l(fs_1s_2s_1)=M-1$, then $l(fs_1s_2s_1s_2)=M$, so that these elements
are in contradiction with Claim~\ref{M2n-1}. This implies that $l(fs_1s_2s_1)=M-3$, see Fig~\ref{cl3}.

Now, suppose that $m_{13}\ne \infty$. Then $f,fs_1\in I$ together with $l(f)=l(fs_1)+1=M$ implies that $fs_1s_3\in I$.
In view of Claim~\ref{M3n-2}, we see that $l(fs_1s_3)\ne M$, so $l(fs_1s_3)=M-2$. This contradicts Lemma~\ref{lemma b}
(applied to the elements $fs_1s_3,fs_1,fs_1s_2,fs_1s_2s_1$ of lengths $M-2,M-1,M-2,M-3$). The contradiction shows that
$m_{13}=\infty$.

\begin{figure}[!h]
\begin{center}
\psfrag{l}{\scriptsize $l$} 
\psfrag{k}{\scriptsize $k$} 
\psfrag{f}{\scriptsize $f$} 
\psfrag{fs1}{\scriptsize $fs_1$} 
\psfrag{fs12}{\scriptsize $fs_1s_2$} 
\psfrag{fs13}{\scriptsize $fs_1s_3$} 
\psfrag{fs121}{\scriptsize $fs_1s_2s_1$} 
\psfrag{fs1212}{\tiny $fs_1s_2s_1s_2$} 
\psfrag{fs213}{\tiny $fs_2s_1s_3$} 
\psfrag{fs2132}{\tiny $fs_2s_1s_3s_2$} 
\psfrag{fs2121}{\tiny $fs_2s_1s_2s_1$} 
\psfrag{fs2123}{\tiny $fs_2s_1s_2s_3$} 
\psfrag{fs21232}{\tiny $fs_2s_1s_2s_3s_2$} 
\psfrag{M}{\scriptsize $M$} 
\psfrag{M1}{\scriptsize $M-1$} 
\psfrag{M2}{\scriptsize $M-2$} 
\psfrag{M3}{\tiny $M-3$} 
\psfrag{M4}{\tiny $M-4$} 
\epsfig{file=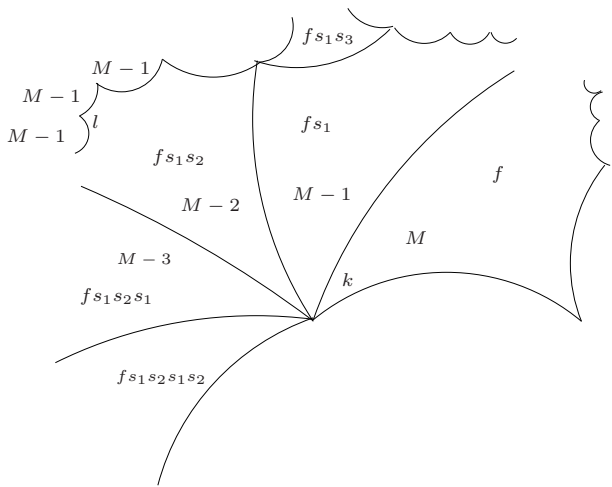,width=0.5\linewidth}
\caption{Towards the proof of Claim~\ref{M1n}}  
\label{cl3}
\end{center}
\end{figure}

Now, consider the neighbors of $fs_1s_2$. Since $l(fs_1s_2s_1)=M-3$ and $m_{1j}=\infty$ for $j>2$, we have $l(fs_1s_2s_j)=M-1$.

Suppose that $fs_1s_2s_3\notin I$. Then $fs_1s_3\in I$ (since $m_{23}\ne \infty$ and $fs_1s_2,fs_1\in I$). Moreover, 
$l(s_1s_3)=M$ (otherwise by Lemma~\ref{lemma7} we have a contradiction with the assumption  $fs_1s_2s_3\notin I$).
However, this contradicts the statement of Claim~\ref{M3n-2}. Hence, $fs_1s_2s_3\in I$. 

Now we use induction on $j$ to show that $fs_1s_2s_j\in I$ for all $3\le j\le n-1$. Indeed, if $fs_1s_2s_{j-1}\in I$ but $fs_1s_2s_j\notin I$, then $fs_1s_2s_{j-1}s_j\in I$, and $l(fs_1s_2s_{j-1}s_{j})=l(fs_1s_2s_{j-1})+1=M$, which contradicts Claim~\ref{M3n-2} or Claim~\ref{M2n-1}.

{ In particular, we obtain that $fs_1s_2s_{n-1}\in I$. Since $fs_1s_2\in I$ and $m_{n-1,n}$ is not divisible by $3$, this implies that at least five elements obtained from $fs_1s_2$ by multiplication by $s_{n-1}$ and $s_n$ belong to $I$. Recall that the lengths of $fs_1s_2s_{n},fs_1s_2,fs_1s_2s_{n-1}$ are $M-1, M-2, M-1$ respectively. Thus, $fs_1s_2s_n,fs_1s_2s_ns_{n-1}, fs_1s_2s_{n-1}s_n\in I$ with $l(fs_1s_2s_ns_{n-1})=l(fs_1s_2s_{n-1}s_n)=M$.}
However, this is in contradiction with the statement of Claim~\ref{M2n-1}, which completes the proof of the claim.
\end{proof}

\medskip
\noindent
Now, combining the results of Claims~\ref{M3n-2}--\ref{M1n} we obtain the lemma.\end{proof}

\section{Minimal groups without finite index reflection subgroups}
\label{min}
In this section we combine results of previous sections to obtain a criterion for odd-angled Coxeter groups to contain no finite index reflection subgroups. 
For this we list
{\it minimal} groups containing no finite index reflections subgroups
(where $W$ is {\it minimal} if $W$ has no finite index  reflection subgroup while each proper special subgroup of $W$ has one). According to Corollary~\ref{subd}, every Coxeter group with connected divisibility diagram containing such a minimal group as a special subgroup has no  finite index reflection subgroup either.

The results will be formulated in terms of divisibility diagrams $\Cox_{\div}(W)$ of odd-angled Coxeter systems $(W,S)$ (see Definition~\ref{div}). 


\begin{theorem}
\label{minimal}
Let $W$ be a minimal odd-angled Coxeter group containing no finite index reflection subgroup, and assume that $\Cox_{\div}(W)$ is connected.
Then $\Cox_{\div}(W)$ is one of the diagrams shown in Table~\ref{answer}.

\end{theorem}

\begin{remark}
It is easy to see that the diagram $\Cox_{\div}(W)$ is not connected (i.e., $S=S_1\sqcup S_2$ { for some nonempty $S_1$ and $S_2$}, and there are no edges between $S_1$ and $S_2$) if and only if $W$ is the free product of the groups generated by $S_1$ and $S_2$.  

\end{remark}

\begin{center}
\begin{table}
\caption{Minimal groups without finite index reflection subgroups and with connected $\Cox_{\div}(W)$}
\label{answer}
\begin{tabular}{|c|c|}
\hline
&\\
$\rank(W)$ & $\Cox_{\div}(W)$ \\
&\\
\hline
&\\
3&  
\begin{tabular}{cp{1cm}cp{0.8cm}c}
\psfrag{k}{\scriptsize $k$}
\psfrag{l}{\scriptsize $l$}
\psfrag{m}{\scriptsize $m,\infty$}
\psfrag{5}{\scriptsize $5$}
\epsfig{file=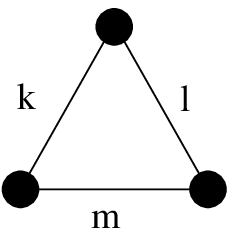,width=0.1\linewidth}
&&
\psfrag{k}{\scriptsize $p$}
\psfrag{l}{\scriptsize $l$}
\psfrag{m}{}
\psfrag{5}{\scriptsize $5$}
\epsfig{file=./pic/cox_klm.eps,width=0.1\linewidth}
&&
\raisebox{0.8cm}{\begin{tabular}{c}
$5\le k,l,m\le\infty$\\
$7\le p\le\infty$
\end{tabular}}
\end{tabular}
\\
\hline
&\\
4 & 
\begin{tabular}{cp{0.65cm}cp{0.8cm}c}
\psfrag{k}{\scriptsize $k$}
\psfrag{l}{\scriptsize $l$}
\psfrag{p}{\scriptsize $3,\infty$}
\psfrag{5}{\scriptsize $5$}
\psfrag{t}{\scriptsize $3,5,\infty$}
\begin{tabular}{c}
\epsfig{file=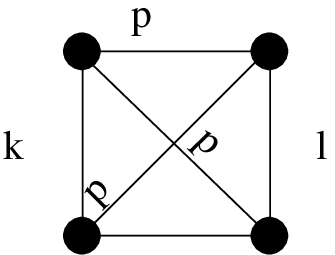,width=0.14\linewidth}\\
$5\le k,l<\infty$\\
\end{tabular}
&&
\psfrag{t}{\scriptsize $3,5,\infty$}
\psfrag{x5}{\scriptsize }
\psfrag{y5}{\scriptsize $5$}
\raisebox{-0.6cm}{\epsfig{file=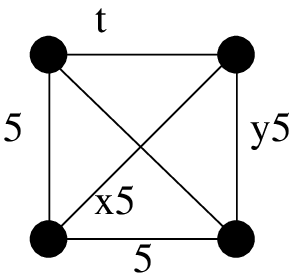,width=0.12\linewidth}}
&&
\psfrag{t}{\scriptsize }
\psfrag{x5}{\scriptsize $5$}
\psfrag{y5}{\scriptsize }
\raisebox{-0.6cm}{\epsfig{file=./pic/cox_5553n.eps,width=0.12\linewidth}}
\end{tabular}
\\
\hline
&\\
$5$ &  
\psfrag{k}{\scriptsize $k$}
\psfrag{l}{\scriptsize $l$}
\psfrag{p}{\scriptsize $3,\infty$}
\epsfig{file=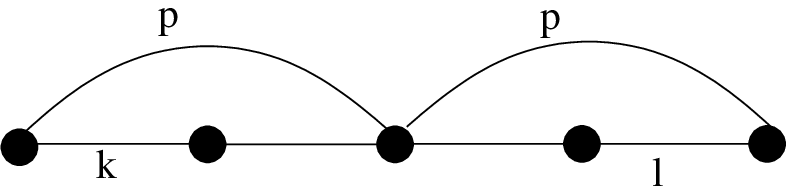,width=0.3\linewidth}
\qquad \qquad $5\le k,l <\infty$
\\
&\\
\hline
&\\
$n\ge 5$ &  
\psfrag{k}{\scriptsize $k$}
\psfrag{l}{\scriptsize $l$}
\psfrag{p}{\scriptsize $3,\infty$}
\epsfig{file=./pic/cox_k3-3l.eps,width=0.6\linewidth}
\qquad \qquad $5\le k,l <\infty$
\\
&\\
\hline

\end{tabular}
\end{table}
\end{center}

\begin{proof}[Proof of Theorem~\ref{minimal}]
We proceed by increasing the rank of $W$.

\medskip
\noindent
{\bf 1. $\rank(W)\le 2$.}
All these groups are either finite or infinite dihedral, so
any of them  has a finite index reflection subgroup.

\medskip

\noindent
{\bf 2. $\rank(W)=3$.} 
The diagrams listed in the table can be expressed as all connected diagrams of order $3$ except the ones shown in Fig.~\ref{3alt}. Now the statement follows from Lemmas~\ref{triangles},~\ref{55inf} and the examples constructed in Section~\ref{examples}. 
\medskip

\begin{figure}[!h]
\begin{center}
\psfrag{k}{\scriptsize $k$}
\psfrag{l}{\scriptsize $l$}
\psfrag{m}{\scriptsize $m,\infty$}
\psfrag{5}{\scriptsize $5$}
\psfrag{k}{\scriptsize $k\ge 3$}
\epsfig{file=./pic/cox_k.eps,width=0.1\linewidth}
\qquad\qquad\qquad 
\epsfig{file=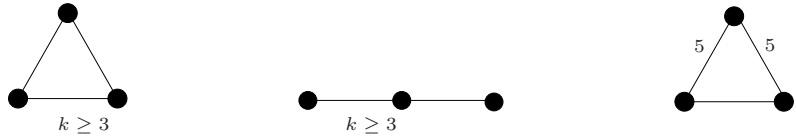,width=0.17\linewidth}
\qquad\qquad\qquad
\raisebox{0.26cm}{\epsfig{file=./pic/cox_553.eps,width=0.1\linewidth}}
\caption{Rank $3$ Coxeter groups having finite index subgroups}  
\label{3alt}
\end{center}
\end{figure}

\noindent
{\bf 3. $\rank(W)=4$.} 
All the diagrams of order $4$ shown in Table~\ref{answer} correspond to groups having no  finite index reflection subgroups (see Lemmas~\ref{555-4} and~\ref{remaining4}),
and no proper subdiagram of these diagrams defines a group having no finite index reflection subgroup.

To prove that the list is complete, note that by Lemma~\ref{k} the diagram $\Cox_{\div}(W)$ of minimal group $W$ should contain at least two multiple edges.
If there are two multiple edges incident to one vertex and at least one of them is labeled other than $5$, then $W$ is not minimal (see rank $3$ case).

So, either there are exactly two multiple edges, and they have no common vertex (as in the left diagram in Table~\ref{answer}, second row), or all multiple edges are labeled by $5$.

Suppose that all multiple edges are labeled by $5$ and there exist two multiple edges incident to the same vertex.
Due to the rank $3$ classification, in any order $3$ subdiagram containing two multiple edges the third edge should be simple. Thus, if there are three multiple edges incident to one vertex, we obtain the diagram shown in Table~\ref{answer} (right), otherwise we  obtain the diagram in the middle of row $2$ in Table~\ref{answer} (according to Lemma~\ref{55}, in this case we need at least three multiple edges).

\medskip
\noindent
{\bf 4. $\rank(W)\ge 5$.} 
By Lemma~\ref{k3-3l}, the diagrams shown in Table~\ref{answer} define Coxeter groups containing no finite index reflection subgroups.
By Lemma~\ref{k}, these diagrams are minimal connected diagrams defining groups without finite index reflection subgroups.
We are left to prove that the list is complete.

The proof is by induction on the rank of $W$. 
Suppose that  $\Cox_{\div}(W)$ is a connected diagram of a minimal group containing no finite index reflection subgroup. 
By Lemma~\ref{k}, $\Cox_{\div}(W)$ contains at least two multiple edges.
It follows from minimality and from the results for smaller ranks that $\Cox_{\div}(W)$ contains at most two multiple edges (if there are three multiple edges it is always possible to find a connected subdiagram containing only two of them { defining a group without finite index reflection subgroups}).
It also follows from minimality that in case of ${\mathrm{rk}}(W)=n$ the vertices of two distinct multiple edges should not be connected in  $\Cox_{\div}(W)$ by a path containing less than $n-3$ edges. 
This proves the completeness of the list.
\end{proof}

So far we considered odd-angled Coxeter groups with connected divisibility diagrams only. Now, suppose that $\Cox_{\div}(W)$ has several connected components, $\Cox_{\div}(W_1),\dots,\Cox_{\div}(W_k)$, where $W_i$ are odd-angled Coxeter groups. In other words, $W=W_1* W_2*\dots* W_k$.

\begin{lemma}
\label{component}
Let $W$ be an odd-angled Coxeter group and $W=W_1* W_2*\dots* W_k$. 
Then  $W$ contains a  finite index reflection subgroup if and only if at least one of $W_1,\dots,W_k$ contains one. 

\end{lemma}

\begin{proof}
First, suppose that $V_1\subset W_1$ is a finite index reflection subgroup. Let $F_{V_1}$ be its principal fundamental domain in the Davis complex of $W_1$, and define $I_1$ to be $\{w \in W_1 | D(w) \in F_{V_1} \}$. Consider now $I_1$ as a set of elements of $W$, and define the polytope $P=\{D(w) | w \in I_1\}$. We want to show that $P$ is a Coxeter polytope.

The facets of $P$ (more precisely, the reflections stabilizing the walls containing the facets) can either belong to $W_1$ or not, so we need to look at different types of dihedral angles of $P$. If both facets defining an angle of $P$ belong to $W_1$, then the value of the dihedral angle is the same as in $F_{V_1}$, so it is an integer part of $\pi$. If one of the facets does not belong to $W_1$, then the other does not belong to $W_1$ either: by assumption on connected components of $\Cox_{\div}W$, no wall of $W_i$ meets any wall not belonging to $W_i$. 
Since all the walls that intersect the interior of $P$ correspond to reflections in $W_1$, in this case there exists a chamber containing the intersection of the facets and  whose defining walls contain the two facets. Hence the dihedral angle between the two facets is a dihedral angle of a chamber, i.e. an integer submultiple of $\pi$. Therefore, all the angles of $P$ are integer submultiples of $\pi$, hence $P$ is a fundamental domain of a finite index reflection subgroup of $W$. 

Now suppose $V\subset W$ is a finite index reflection subgroup. Since $V$ is a proper subgroup of $W$, the principal fundamental domain $F_V$ of $V$ contains a chamber $D(s)$ for some generating reflection $s\in S$. Clearly, there is a unique $W_i$ containing $s$. The same { reasonings as the ones} used in the proof of Lemma~\ref{special subgroup} show that $W_i$ has a finite index reflection subgroup.  
\end{proof}

Now we combine Theorem~\ref{minimal} with Lemma~\ref{component} to obtain the following criterion. 

\begin{theorem}
\label{cor1}
An odd-angled Coxeter group $W$ contains no finite index reflection subgroup if and only if each connected component of $\Cox_{\div}(W)$ contains one of the diagrams shown in Table~\ref{answer} as a subdiagram.

\end{theorem}

We close the paper with the following reformulation of Theorem~\ref{cor1} not referring to Table~\ref{answer}.

\begin{cor}
\label{cor2}
An odd-angled Coxeter group $W$ contains a finite index reflection subgroup if and only if 
$\Cox_{\div}(W)$ contains at least one connected component $C$ of one of the following three types:

\begin{itemize}
\item[1.] the order of $C$ is $1$ or $2$;

\item[2.] $C$ contains at most one multiple edge;

\item[3.] $C$ contains a subdiagram $D$ of order $3$ with labels $(5,5,3)$, and every non-absent edge of $C$ except the edges of $D$ is simple.

\end{itemize}

\end{cor}

{
\section*{Acknowledgments}

A.F. and P.T. would like to thank IHES in Bures-sur-Yvette for hospitality and great research atmosphere during their stay. A.F. is also grateful to the Max Planck Institute for Mathematics in Bonn for the financial support and excellent research conditions. An essential part of the work was completed while all the authors were affiliated with Jacobs University Bremen.  The authors are also grateful to the anonymous referees for valuables comments and remarks. 

}

\end{document}